\documentclass[final,3p,times,11pt]{elsarticle}
\usepackage[fleqn]{amsmath}
\usepackage{float}
\usepackage{amssymb}
\usepackage{amsthm,mathrsfs}
\usepackage{subfigure}
\usepackage{lineno}
\usepackage{amsfonts}
\usepackage{epsfig}
\usepackage{color}
\usepackage{setspace}
\usepackage{latexsym}
\usepackage[colorlinks,citecolor=blue,urlcolor=blue,linkcolor=blue]{hyperref}
\usepackage{graphicx}
\usepackage{float}
\usepackage{exscale,relsize}
\allowdisplaybreaks[4]
\DeclareMathOperator{\Tr}{Tr}

\biboptions{sort&compress}

\newtheorem{thm}{Theorem}[section]
\newdefinition{pro}{proposition}[section]
\newdefinition{cor}{Corollary}[section]
\newdefinition{lem}{Lemma}[section]
\newdefinition{definition}{Definition}[section]\newdefinition{rem}{Remark}[section]
\newdefinition{ass}{Assumption}[section]
\newdefinition{prop}{Proposition}[section]
\newdefinition{exam}{Example}[section]

\begin{document}

\begin{frontmatter}
\title{Universal bifurcation scenarios in delay-differential equations with one delay\tnoteref{t1}}

\tnotetext[t1]{This work was supported by the Deutsche Forschungsgemeinschaft (DFG, German Research Foundation), Project No. 411803875.}

\author[labela,labelb,labelc]{Yu Wang}
\author[labela]{Jinde Cao} 
\author[labelb,labele]{Jürgen Kurths}
\author[labelf,labelb]{Serhiy Yanchuk\corref{cor1}}\ead{syanchuk@ucc.ie}

\address[labela]{School of Mathematics, Southeast University,  210096 Nanjing, China}
\address[labelb]{Potsdam Institute for Climate Impact Research,  14473 Potsdam, Germany}
\address[labelc]{Institute of Mathematics,  Humboldt-Universität zu Berlin,  10099 Berlin,  Germany}
\address[labele]{Department of Physics, Humboldt-Universität zu Berlin,  10099 Berlin, Germany}
\address[labelf]{University College Cork, School of Mathematical Sciences, Western Road, Cork, T12 XF62, Ireland}
\cortext[cor1]{Corresponding author}
\journal{Journal of Differential Equations}

\begin{abstract}
    We show that delay-differential equations (DDE) exhibit universal bifurcation scenarios, which are observed in large classes of DDEs with a single delay. Each such universality class has the same sequence of stabilizing or destabilizing Hopf bifurcations. These bifurcation sequences and universality classes can be explicitly described by using the asymptotic continuous spectrum for DDEs with large delays. 
   Here, we mainly study linear DDEs, provide a general transversality result for the delay-induced bifurcations, and consider three most common universality classes. 
    For each of them, we explicitly describe the sequence of stabilizing and destabilizing bifurcations. We also illustrate the implications for a nonlinear Stuart-Landau oscillator with time-delayed feedback. 
\end{abstract}   
\end{frontmatter}

\section{Introduction}\label{sec1}
Delay-differential equations (DDEs)  are important mathematical models in many application areas, including optics \cite{Mallet-Paret1989,Erneux2009, Atay2010a,Soriano2013,Yanchuk2019}, physiology and infections disease modelling \cite{Kuang1993, Hartung2006a, Diekmann2007, Muller2015, Agaba2017, Young2019},  mechanics \cite{Tsao1993, Insperger2006, InspergerStepan2011, Otto2013, Otto2015}, neuroscience \cite{Wu2001, Izhikevich2006, Ko2007}, and others. 
More recently, DDEs have become a focus in some areas of machine learning, such as reservoir computing  \cite{appeltantInformationProcessingUsing2011, Grigoryeva,  Larger2017,  Keuninckx2017,  Stelzer2019, Koster2022b} or deep neural networks \cite{Stelzer2021a,Furuhata2021}.

Here we study how time-delay affects the stability of equilibria. This problem commonly arises in the above mentioned applications when the time-delay is a parameter. 
Although the equilibrium of a DDE remains unchanged upon variation of the time-delay, its stability may change. The recent work \cite{Yanchuk2022b} provides the necessary and sufficient conditions for DDEs with multiple delays to be absolutely stable or hyperbolic, i.e., its equilibrium does not have any destabilizing or stabilizing bifurcations as time-delay changes, see also the works \cite{Pontryagin1942, Elsgolz1971a, Brauer1987, Boese1995, Smith2010, Li2016, An2019, Yanchuk2022b} related to this property. This behavior can also be called delay-independent stability.

The DDEs that are not absolutely stable in the sense of \cite{Yanchuk2022b} exhibit sequences of stabilization or destabilization bifurcations as time-delay changes. 
Often, these bifurcations lead to complex dynamical behaviors \cite{ Ikeda1982, Culshaw2003, Guo2008, Erneux2009, Yanchuk2017, Al-Darabsah2020} or to an increasing coexistence of multiple periodic solutions \cite{Yanchuk2009}. 
Moreover, the bifurcation scenarios seem to be similar in many cases, and it seemed to us that different studies had to reinvent the same bifurcation scenarios in different systems over and over again.
We show here that the majority of DDEs with one delay indeed have only  a few bifurcation scenarios, and these scenarios can be explicitly described. 
This leads to the main idea of this paper: introducing the universality classes of linear DDEs, such that each universality class has the same behavior of the critical eigenvalues, which can be explicitly described once for the whole class. 
In a sense, the absolutely hyperbolic DDEs from \cite{Yanchuk2022b} would represent the universality class ``0". 

Our classification relies on the notion of \textit{asymptotic continuous spectrum} (ACS) for DDEs, which was originally introduced for DDEs with large delays \cite{Wolfrum2006,  Yanchuk2010a, Lichtner2011, Sieber2013, Yanchuk2015b, Ruschel2021, poignard_self-induced_2022}, but, as we show here, it plays a crucial role for describing the bifurcation scenarios for arbitrary delays. An important feature of ACS is that it is rather easy to calculate.

The structure of this manuscript is as follows: Section \ref{sec2}  introduces the main notations, including the asymptotic continuous spectrum. 
Section \ref{sec3} gives necessary and sufficient conditions for the occurrence of delay-induced Hopf bifurcations. The transversality theorem gives explicit conditions for the transversality and the direction of change of the critical roots of the characteristic equation with the change of the time delay.
In sections \ref{sec4} to \ref{sec6}, we introduce three basic universality classes of linear DDEs and give necessary and sufficient conditions for a DDE to belong to each class. 
Furthermore, for each class, we give explicit results on the existence and transversality of the critical roots of the characteristic equation. The theory is illustrated with several examples.
In section \ref{sec7}, we show the implications of our results for an example of a nonlinear Stuart-Landau oscillator with time-delayed feedback. A brief discussion and possible further extensions are presented in section \ref{sec8}.

\section{Model and notations}\label{sec2}

We start with the general nonlinear delay differential equation (DDE)
\begin{align}
  \dot x(t)=\mathbf{F}\left(x(t),x(t-\tau)\right), \label{eq:nonlinearDDE}  
\end{align}
where  $\mathbf{F}: \mathbb{R}^n\times \mathbb{R}^n \mapsto \mathbb{R}^n$ is continuously differentiable in its arguments, $x(t)\in\mathbb{R}^n$, and $\tau > 0$ is the time delay.
When studying the linear stability of equilibria the DDE \eqref{eq:nonlinearDDE}, the linearized system 
\begin{align}
     \dot{x}(t)=Ax(t)+Bx(t-\tau)\label{eq:DDE}
\end{align}
is obtained, where  $A = \frac{\partial \mathbf{F}}{\partial x(t)}, B = \frac{\partial \mathbf{F}}{\partial x(t-\tau)} \in\mathbb{R}^{n\times n}$. Although the main subject of this study is the linear DDE \eqref{eq:DDE}, we also illustrate the implications of our results for a nonlinear example in Sec.~\ref{sec7}.

The stability problem for \eqref{eq:DDE} can be reduced to the characteristic quasipolynomial \cite{Hale1993,Smith2010}
\begin{align}
\chi(\lambda) = \det\left[\lambda I-A-Be^{-\lambda\tau}\right]=0.\label{eq:CharEq}
\end{align}
With changing time-delay $\tau\ge 0$, the number of unstable characteristic roots of DDE \eqref{eq:DDE} can change due to the following mechanism.

\begin{definition}[Delay-induced transverse crossing]
    The DDE \eqref{eq:DDE} undergoes \textit{delay-induced transverse crossing} at $\tau=\tau_H$ if there exists a family of roots $\lambda(\tau) = \alpha(\tau) \pm i\mu(\tau)$ of the characteristic equation \eqref{eq:CharEq}  such that $\alpha(\tau_H)=0$ and $\frac{d\alpha(\tau_H)}{d\tau}\ne 0$. 
\end{definition}

We also distinguish between stabilizing and destabilizing transverse crossing.
\begin{definition}[Stabilizing and destabilizing transverse crossing]
    The delay-induced transverse crossing is \textit{stabilizing} if 
    $\frac{d\alpha(\tau_H)}{d\tau} <  0$, and \textit{destabilizing} if $\frac{d\alpha(\tau_H)}{d\tau} > 0$. 
\end{definition}
 
The following definitions are known from the theory of DDEs with large delays \cite{Wolfrum2006,  Lichtner2011, Sieber2013, Yanchuk2015b, Yanchuk2022b}. To provide motivation and a better intuition for these definitions, we apply the following Ansatz
\begin{align*}
\lambda=\frac{\gamma(\omega)}{\tau}+i\omega,
\end{align*}
with $\gamma,\omega\in \mathbb{R}$ in the characteristic equation (\ref{eq:CharEq})
\begin{align}
\det\left[\left(\frac{\gamma(\omega)}{\tau}+i\omega\right) I-A-Be^{-\gamma(\omega)-i\omega\tau}\right]=0.\label{Eq3}
\end{align}
Denoting $\phi(\omega)=\omega\tau$ and assuming a large time-delay $\gamma\ll \tau$, we get  an approximate equation 
\begin{align}
\det\left[i\omega I-A-Be^{-i\phi(\omega)}\right]=0,\label{eq:noname}
\end{align}
which motivates the following definition.

\begin{definition}[Generating polynomial]
We define
\begin{align}
p_\omega(Y) := \det\left[i\omega I-A-BY\right]\label{eq:Y}
\end{align}
to be the \textit{generating polynomial}. 
\end{definition}
Note that $p_\omega(Y)$ is a polynomial in $Y$ and $\omega$ and also $p_\omega(e^{-i\phi(\omega)})=0$ is equivalent to the approximate characteristic equation \eqref{eq:noname}. 

Next, we give the definition of the asymptotic continuous spectrum, i.e., the curves that determine the DDE spectrum for large delays. As we will show in this work, this spectrum plays a crucial role in describing transverse crossings and their type for any finite time-delays.

\begin{definition}[Asymptotic Continuous Spectrum, ACS]
\label{def:1}
The \textit{asymptotic continuous spectrum (ACS)} is given by
\begin{align}
\lambda_\text{ACS}(\omega) & = \left\{ \frac{1}{\tau}\gamma_j(\omega) + i\omega\ :\ \mathbb{R}\mapsto \mathbb{C}, \quad j=1,\dots,m\right\}, \label{eq:ACS-complex}\\
 & \text{where} \quad \gamma_{j}(\omega)  =-\ln\left|Y_{j}(\omega)\right|,\label{eq:ACS}
\end{align}
and $Y_{j}(\omega)$ are roots of the generating polynomial (\ref{eq:Y}) $p_\omega(Y_j(\omega))=0$. The points $\omega$ are excluded, for which the polynomial is degenerate with $p_\omega(Y)\equiv \det\left[i\omega I-A\right]$ for all $Y$. 
\end{definition}
The computation of the ACS is reduced to a simple polynomial root finding with a degree less than or equal to the number of components in the DDE. Note also that the spectrum is symmetric with respect to the complex conjugation for DDEs with real coefficients. The ACS is completely described by the real-valued functions $\gamma_j(\omega)$, which determine the (rescaled) real part of the spectrum.

\section{Critical characteristic roots and transversality theorem}\label{sec3}

One of the most important properties of the ACS is that the roots $\omega_H$ of the functions $\gamma_j(\omega)$ are the only possible crossing frequencies, when considering $\tau$ as a parameter. More precisely, the following transversality theorem holds.

\begin{thm}[\textbf{Transversality Theorem}] \label{thm:Trans}
    Let $\lambda_c=i\omega_H$ $(\omega_H>0)$ be a critical simple characteristic root of the linear DDE \eqref{eq:DDE} for $\tau=\tau_H$. Then there is a delay-induced transverse crossing of this characteristic root at $\tau=\tau_H$ if and only if there is a branch of ACS with $\gamma(\omega_H)=0$ and $\frac{d}{d\omega}\gamma(\omega_H)\ne 0$. Moreover,  
    \begin{itemize}
        \item if $\frac{d}{d\omega}\gamma(\omega_H)<0$, then the  crossing is destabilizing with $\frac{d}{d\tau} \Re [\lambda(\tau_H)]>0$,
    \item if $\frac{d}{d\omega}\gamma(\omega_H)>0$, then the crossing is stabilizing with  $\frac{d}{d\tau} \Re [\lambda(\tau_H)]<0$.
    \end{itemize}
\end{thm}

\begin{proof}
Firstly we note that $\lambda_c=i\omega_H$ being the characteristic root is equivalent to
\begin{equation}
    0 = \det\left[i\omega_H I -A-Be^{-i\omega_H \tau_H}\right]=
    \det\left[i\omega_H I -A-Be^{-i\phi_H}\right] = p_{\omega_H}\left(e^{-i\phi_H}\right),
\end{equation}
or $Y(\omega_H)=e^{-i\phi_H}$ and $\gamma(\omega_H)=0$. The existence of a nontrivial branch $Y(\omega)$ of ACS follows from the assumption that $\lambda_c$ is simple. Hence the necessary condition holds. Let us show the sufficient condition.

Using the following notation
\begin{equation}
\label{eq:PzY}
P(z,Y)=\det\left[z I -A-BY\right]=0
\end{equation}
and by differentiating the characteristic quasipolynomial \eqref{eq:CharEq}, we get
\[
\partial_{\tau}\lambda=-\frac{\partial_{\tau}\chi}{\partial_{\lambda}\chi}=-\frac{-\lambda e^{-\lambda\tau}\partial_{Y}P}{\partial_{z}P-\tau e^{-\lambda\tau}\partial_{Y}P}.
\]
The denominator of the obtained expression is nonzero for a simple root $\lambda=i\omega_H$.
Next, we evaluate the above expression at $\lambda_c=i\omega_{H}$ and  the time-delay $\tau=\tau_H$
\[
\left.\partial_{\tau}\lambda_c\right|_{\lambda_c=i\omega_{H},\tau=\tau_{H}}=-\frac{-i\omega_{H}e^{-i\phi_{H}}\partial_{Y}P}{\partial_{z}P-\tau_{H}e^{-i\phi_{H}}\partial_{Y}P}=\frac{i\omega_{H}}{\frac{\partial_{z}P}{\partial_{Y}P}e^{i\phi_{H}}-\tau_{H}}
=
-\frac{i\omega_{H}}{\left.\partial_{z}Y'\right|_{z=i\omega_{H}}e^{i\phi_{H}}+\tau_{H}},
\]
where $\phi_H=\omega_H \tau_H$ and $Y'$ is a root of the polynomial \eqref{eq:PzY}, i.e. $p(z,Y')=0$. Note that $Y(\omega)=Y'(i\omega)$ is the root of the generating polynomial \eqref{eq:Y}. 
Using $\partial_{\omega}Y=i\partial_{z}Y'$, we have
\[
\left.\partial_{\tau}\lambda_c\right|{}_{\lambda_c=i\omega_{H},\tau=\tau_{H}}=-\frac{i\omega_{H}}{-i\left.\partial_{\omega}Y\right|_{\omega=\omega_{H}}e^{i\phi_{H}}+\tau_{H}}.
\]
Further we use $ Y(\omega)=e^{-\gamma(\omega)-i\phi(\omega)}$ and  obtain
\[
\left.\partial_{\omega}Y(\omega)\right|_{\omega=\omega_{H}}=-e^{-i\phi_{H}}\left[\partial_\omega \gamma\left(\omega_{H}\right)+i\partial_\omega \phi\left(\omega_{H}\right)\right],
\]
and hence
\[
\left.\partial_{\tau}\lambda_c\right|_{\lambda_c=i\omega_{H},\tau=\tau_{H}}=-\frac{i\omega_{H}}{\tau_{H}-\partial_\omega\phi\left(\omega_{H}\right)+i\partial_\omega \gamma\left(\omega_{H}\right)}.
\]
Taking the real part
\begin{align}
    \left. \frac{d}{d\tau} \Re [\lambda(\tau) ] 
\right|_{\lambda=\lambda_c,\tau=\tau_H}=
\frac{-\omega_{H}\partial_\omega \gamma\left(\omega_{H}\right)}{\left(\tau_{H}-\partial_\omega \phi\left(\omega_{H}\right)\right)^{2}+\left(\partial_\omega \gamma\left(\omega_{H}\right)\right)^{2}}.\label{eq:dReal}
\end{align}
We observe that the sign of the term $-\omega_{H}\partial_\omega \gamma\left(\omega_{H}\right)$ determines the transversality and the direction of the crossing. 
Since $\omega_{H}>0$, we have $\left. \frac{d}{d\tau} \Re [\lambda(\tau) ] 
\right|_{\lambda=\lambda_c,\tau=\tau_{H}}>0$ for $\partial_\omega \gamma(\omega_{H})<0$ and  $\left. \frac{d}{d\tau} \Re [\lambda(\tau) ] 
\right|_{\lambda=\lambda_c,\tau=\tau_{H}}<0$ for $\partial_\omega \gamma(\omega_{H})>0$.
The proof is complete.
\end{proof}

Note that we are considering the situation where the coefficients of the DDE \eqref{eq:DDE} are real, so the critical eigenvalues appear in pairs $\pm i\omega_H$. Therefore, it is sufficient to consider $\omega_H>0$ to determine the crossing direction.

As the delay-induced transverse crossings can only be caused by the ACS, the following theorem follows from the above result. 

\begin{thm}[\textbf{Universal sequence of crossing events / Bifurcation Theorem}]\label{thm:Hopf}
Let a curve of the ACS crosses the imaginary axis at $i\omega_H$, i.e., $\gamma_j(\pm \omega_H)=0$ for some $j\in\mathbb{N}$ and $\frac{d}{d\omega}\gamma_j(\omega_H)\ne 0$. Then the delay-induced crossings of the characteristic multipliers in system \eqref{eq:DDE} occur with $\lambda_H=i\omega_H$ at the following values of time-delays
\begin{equation}
\label{eq:generalHopf}    
\tau_k = \frac{1}{\omega_H}(\phi_H +2\pi k), \quad k\in \mathbb{Z},
\end{equation}
where 
\begin{equation}
    \label{eq:generalphiH}
    \phi_{H}=-\arg\left[Y_{j}(\omega_{H})\right],
\end{equation}
and $Y_{j}(\omega_{H})$ is a corresponding root of the generating polynomial $p_{\omega_H}(Y)=0$. This crossing is destabilizing, if $\frac{d}{d\omega}\gamma_j(\omega_H)<0$ and  stabilizing if $\frac{d}{d\omega} \gamma_j(\omega_H)>0$.


If there are no other branches of ACS with $\gamma(\omega)=0$, the above mentioned crossings are the only delay-induced transverse crossings in \eqref{eq:DDE}.
\end{thm}
\begin{proof}
    The proof follows from the observation that 
    $$
        \det\left[i\omega_H I -A-Be^{-i\omega_H \tau_k}\right]=
        \det\left[i\omega_H I -A-Be^{-i\phi_H}\right] = p_{\omega_H}\left(e^{-i\phi_H}\right) = p_{\omega_H}(Y_j(\omega_H))=0,
    $$
and the transversality Theorem \ref{thm:Trans}.
\end{proof}

We have shown that a transverse intersection of the critical characteristic roots with changing $\tau$ is only possible at the intersection points of ACS with the imaginary axis. 
Besides the critical roots mentioned in Theorem~\ref{thm:Hopf}, other types of critical roots can appear in system \eqref{eq:DDE}, but they are all not delay-induced, and they do not change with time delay $\tau$ and do not lead to bifurcations. For example, if $i\omega\in \sigma(A)$ with the eigenvector $v_\omega$ and $v_\omega\in\ker B$, then $i\omega$ is the non-delay-induced characteristic root of the quasipolynomial \eqref{eq:CharEq} for all time delays $\tau$. It is an interesting open question whether other cases are possible, but we do not address this question in this paper. 

\section{Universality class I}\label{sec4} 

Starting from this section, we will define universality classes for linear DDEs that exhibit  universal bifurcation scenarios. These universality classes are determined solely by the qualitative features of their ACS. In particular, the ACS of class I is shown schematically in Fig.~\ref{fig:classI}. A curve of such a spectrum crosses the imaginary axis at two points $\pm i\omega_H$ and has a single unstable component for $\omega \in (-\omega_{H},\omega_H)$, while the remaining parts of the ACS are stable. Such a situation is common in applications \cite{Mackey1977a,Weicker2012,Yanchuk2015b,Yanchuk2017}.
Note that the classification is given for linear DDEs. Given a nonlinear DDE, it is natural to apply the classification to a particular equilibrium of a nonlinear system and its linearization. 
\begin{figure}[H]
		\centering		
        \subfigure{\includegraphics[width=4.0cm]{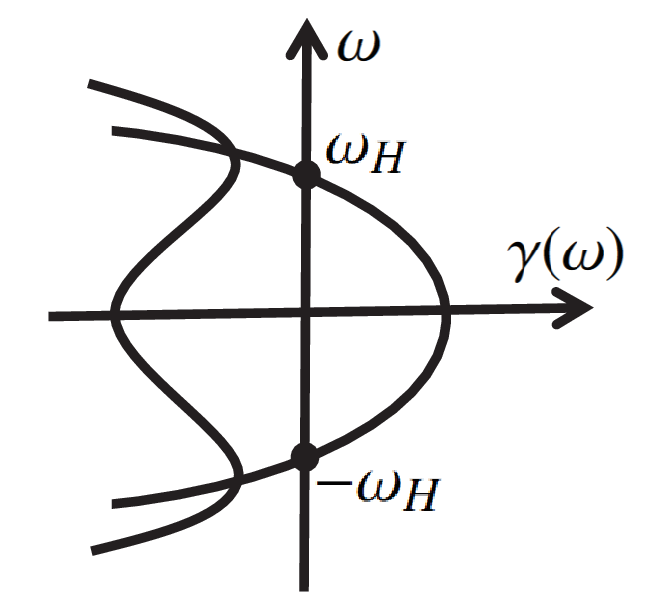}}
		\caption{Schematic representation of class I asymptotic continuous spectrum (ACS). $\lambda=\pm i\omega_H$ are the only possible critical characteristic roots.
        }
		\label{fig:classI}
	\end{figure}

\begin{definition}[Class I asymptotic spectrum] \label{as:uc1} 
    We define the asymptotic spectrum to be of \textit{universality class I ACS} if there exist $\omega_{H}>0$ and $j\in \mathbb{N}$ such that one branch of the ACS satisfies $\gamma_{j}\left(\pm\omega_{H}\right)=0$,  $\gamma_{j}(\omega)>0$ for all $\omega \in (-\omega_{H},\omega_H)$, and $\gamma_{j}(\omega)<0$ for all $\omega \not\in [-\omega_{H},\omega_H]$.
    Also, $\gamma_{k}\left(\omega\right)\ne 0$
    for all $k\ne j$ and $\omega\in\mathbb{R}$. 
\end{definition}

\begin{definition}[Class I DDEs]
    We define the linear DDE \eqref{eq:DDE} to be of \textit{universality class I} if it has the ACS of universality class I. 
\end{definition}

The structure of the ACS and the Bifurcation Theorem \ref{thm:Hopf} lead to the following explicit and complete description of the behavior of critical characteristic roots in the class I DDEs.    

\begin{cor}[Bifurcation scenario in class I DDEs] \label{thm:1} 
Assume that a linear DDE \eqref{eq:DDE} is of universality class I. 
Then the system has a sequence of \textit{destabilizing} delay-induced crossings of the critical characteristic roots with $\lambda_c=\pm i\omega_{H}$ for the following time-delays 
\begin{align}
\tau=\tau_{k}=\frac{1}{\omega_{H}}\left(\phi_{H}+2\pi k\right),\quad k=0,1,2,\dots,\label{eq:tauHopf}
\end{align}
where $\phi_{H}$ is defined by Eq.~\eqref{eq:generalphiH}. Moreover, there are no other transverse delay-induced crossings of the characteristic multipliers in this system. 
\end{cor}

In the following, we present several types of DDEs belonging to the universality class I. In particular, we will see that scalar DDEs are always of class I or class 0, where 0 stands for the absolute hyperbolicity \cite{Yanchuk2022b}. 

\subsection{Scalar DDEs}\label{subsec3-1}

We start with the following scalar equation
\begin{equation}
    \dot{x}(t)=ax(t)+bx(t-\tau),\quad a,b\in\mathbb{R}.\label{eq:DDE-scalar}
\end{equation}
The corresponding characteristic equation is 
\begin{equation}
\lambda-a-be^{-\lambda\tau}=0,\label{eq:ce-scalar}
\end{equation}
and the generating polynomial \eqref{eq:Y} $ i\omega-a-bY=0$ is linear
with the single root 
\[
Y=\frac{i\omega-a}{b}.
\]
Hence, the ACS consists of one curve $\lambda_{\text{ACS}}(\omega)=\frac{\gamma(\omega)}{\tau} +i\omega$, where
\begin{equation}
\gamma\left(\omega\right)=-\ln\left|\frac{i\omega-a}{b}\right|=-\frac{1}{2}\ln\frac{a^{2}+\omega^{2}}{b^{2}}.\label{eq:ac-scalar}
\end{equation}
Function \eqref{eq:ac-scalar} is even and has a single maximum $ \gamma_{\max}=\gamma(0)=-\ln\left|\frac{a}{b}\right|$, see Fig. \ref{fig:ACS-scalar}. 
\begin{figure}[H]
		\centering		\subfigure{\includegraphics[width=7cm]{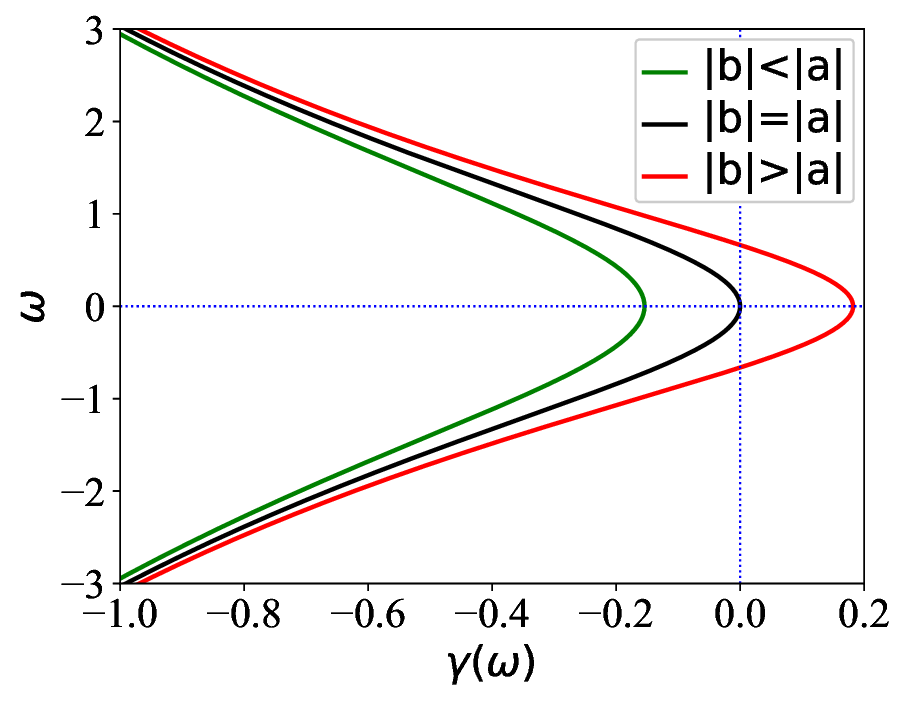}}
		\caption{Asymptotic continuous spectrum for the scalar DDE \eqref{eq:ce-scalar}.
		\label{fig:ACS-scalar}
        }
	\end{figure}
 
In particular, Definition~\ref{as:uc1} is satisfied for $\gamma_{\max}>0$,
or, equivalently, for $|b|>|a|$. In this case, system \eqref{eq:DDE-scalar}
is an example of the  universality class I DDE. 

The value of $\omega_{H}$ is found from the condition $\gamma\left(\omega_{H}\right)=0$, leading to 
\[
\omega_{H}=\sqrt{b^{2}-a^{2}},
\]
where we have taken a positive root.  The phase $\phi_{H}$ is given by Eq.~\eqref{eq:generalphiH} 
\[
\phi_{H}=-\arg\left[\frac{i\omega_{H}-a}{b}\right]=\arg\left[-\frac{a}{b}-i\sqrt{1-\left(\frac{a}{b}\right)^{2}}\right].
\]

According to Corollary~\ref{thm:1}, transverse destabilizing crossings or characteristic roots of \eqref{eq:ce-scalar} with $\lambda_c = \pm i\omega_{H}$ occur for the following time-delays
\begin{align}
 & \tau_{k}=\frac{1}{\omega_{H}}(\phi_{H}+2\pi k)=\nonumber \\
 & =\frac{1}{\sqrt{b^{2}-a^{2}}}\left(\arg\left[-\frac{a}{b}-i\sqrt{1-\left(\frac{a}{b}\right)^{2}}\right]+2\pi k\right),\quad k=0,1,2,\cdots.\label{eq:tauk_scalar}
\end{align}

Figure~\ref{Fig:Scal} shows the spectrum for different values of the critical delays under the condition $|a|<|b|$. 
The expression \eqref{eq:tauk_scalar} leads to the following explicit result for the dimensionality of the unstable manifold for \eqref{eq:DDE-scalar}.

\begin{cor}[Dimension of the unstable manifold in a scalar DDE]
  If $|b|>|a|$, the dimension $D_{\text{u}}$ of the unstable manifold
of the equilibrium in system \eqref{eq:DDE-scalar} is given by the integer number
\begin{equation}
D_{\text{u}}=2\left\lceil \frac{1}{2\pi}\left(\tau\omega_{H}-\phi_{H}\right)\right\rceil +\nu=2\left\lceil \frac{1}{2\pi}\left(\tau\sqrt{b^{2}-a^{2}}-\arg\left[-\frac{a}{b}-i\sqrt{1-\left(\frac{a}{b}\right)^{2}}\right]\right)\right\rceil +\nu,\label{eq:Duu}
\end{equation}
where $\arg(\cdot)$ is the minimal positive argument of the complex
number, $\left\lceil \cdot\right\rceil $ is the ceiling function,
and 
\[
\nu=\begin{cases}
1,\quad\text{ for }a+b>0,\\
0,\quad\text{ for }a+b<0.
\end{cases}
\]

If $|b|<|a|$, it holds $D_{\text{u}}=\nu$. 
\end{cor}
\begin{proof}
    Let us set $\tau=0$. Then obviously $D_{\text{u}}=\nu$. 
Consider now the case $|b|<|a|$. In this case, the ACS does not cross
the imaginary axis and, hence, there are no transverse crossings and $D_{\text{u}}$ does not change with $\tau$. 

In the case $|b|>|a|$, Theorem \ref{thm:1} describes all possible
transverse crossings. No other (de)stabilization is possible, since $\lambda=0$ cannot be the characteristic root for $|b|>|a|$. 
Hence, at each $\tau_{k}$ a pair characteristic roots destabilizes, starting from $\tau_{0}$.
This implies that 
\begin{equation}
D_{\text{u}}=2(k+1)+\nu\quad \text{for}\quad  \tau_{k}<\tau\le\tau_{k+1}.
\label{eq:Du}
\end{equation}
 The latter inequality can be written
as 
\[
\frac{1}{\omega_{H}}\left(\phi_{H}+2\pi k\right)<\tau\le\frac{1}{\omega_{H}}\left(\phi_{H}+2\pi\left(k+1\right)\right)
\]
or
\begin{equation}
k<\frac{1}{2\pi}\left(\omega_{H}\tau-\phi_{H}\right)\le k+1.\label{eq:Du1}
\end{equation}
The inequality \eqref{eq:Du1} together with \eqref{eq:Du} implies
\eqref{eq:Duu}.
\end{proof}

\begin{figure}
\centering
\subfigure{\includegraphics[width=6.2cm]{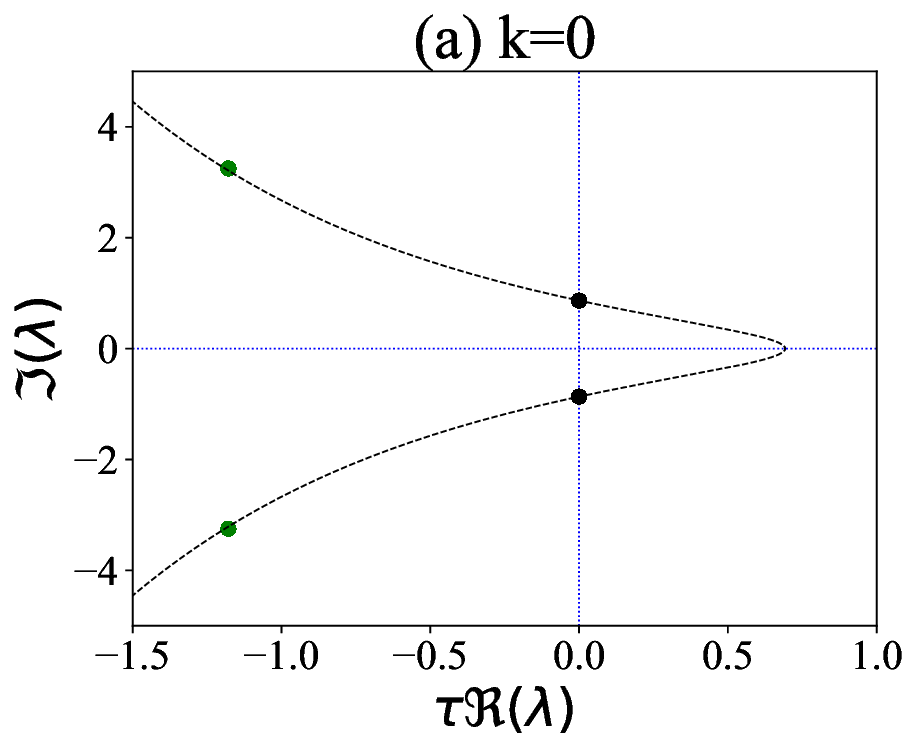}}
\subfigure{\includegraphics[width=6.2cm]{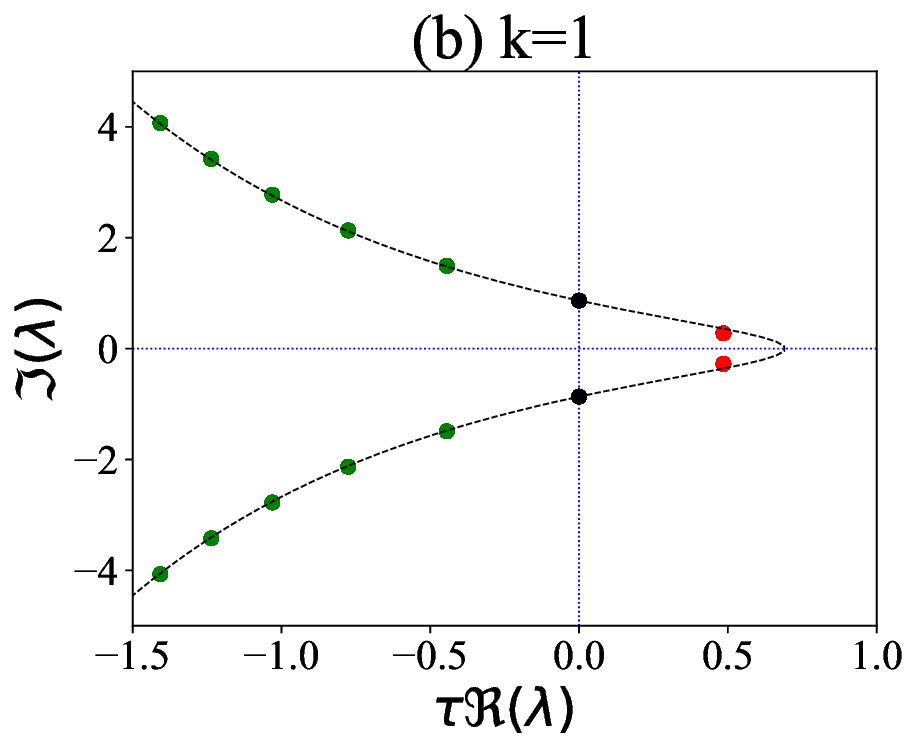}}
\subfigure{\includegraphics[width=6.2cm]{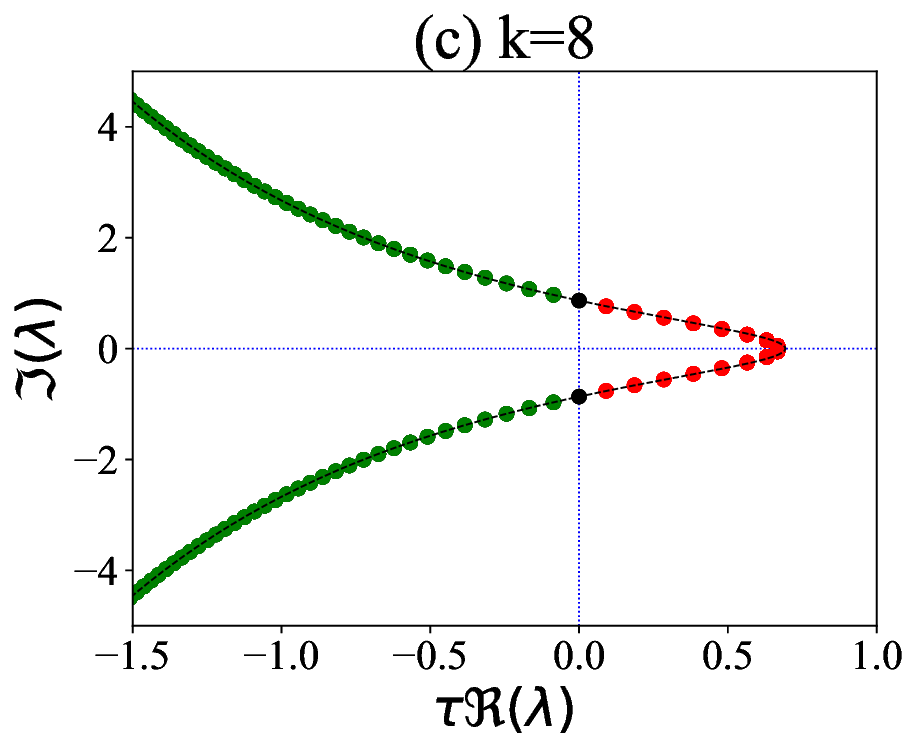}}
\caption{
The spectrum of the scalar DDE \eqref{eq:DDE-scalar} with $a=-0.5$, $b=-1$, and (a) $\tau=\tau_0=2.4184$; (b) $\tau=\tau_1=9.6736$; (c) $\tau=\tau_8=60.4600$, where $\tau_k$ is given by \eqref{eq:tauk_scalar}. The dashed line denotes the curve of the asymptotic continuous spectrum.
The case (a) corresponds to the first pair or eigenvalues becomes critical, (b) the second pair, and (c) to the case when 8 pairs of eigenvalues are unstable and the 9th pair is critical. 
The critical frequency and the phase are $\omega_{H}=0.866$ and $\phi_{H}=2.0944$. 
}
\label{Fig:Scal}
\end{figure}

\subsection{Two-variable linear DDE}\label{subsec3-2}
Here we provide conditions for two-variable DDEs to belong to the universality class I and thus to possess the bifurcation ``scenario" as described in Corollary~\ref{thm:1}.
Consider the two-variable linear equation \eqref{eq:DDE} with 
\begin{align}
A = 
\left(\begin{array}{ll}
a_{11} & a_{12}\\
a_{21} & a_{22}
\end{array}\right),
\quad 
B = \left(\begin{array}{ll}
b_{11} & b_{12}\\
b_{21} & b_{22}
\end{array}
\right),\label{eq:TwovarDDE}
\end{align}
and $a_{ij},b_{ij}\in\mathbb{R},i,j=1,2$. 
The corresponding generating polynomial \eqref{eq:Y} reads 
\begin{equation}
 \det(B)Y^{2}+\left[C-i\omega \Tr(B)\right]Y+\det(A)-i\omega \Tr(A)-\omega^{2}=0,
\label{eq:YTwo}
\end{equation}
where 
\begin{equation}
\label{eq:C}
C=\det(A+B)-\det(A)-\det(B).
\end{equation}

We will consider the cases $\det (B) = 0$ and $\det (B) \ne 0$ separately. First of all we exclude the case $C=\Tr (B) = \det (B) = 0$. Then, the characteristic equation \eqref{eq:CharEq} is reduced to 
\begin{align*}
    \det (A) - \lambda \Tr (A) + \lambda^2 = 0,
\end{align*}
which is simply the characteristic equation for the instantaneous part, and the time-delay has no effect on the spectrum. Therefore, in the following we assume that $C$, $\Tr (B)$, and $\det (B)$ do not vanish simultaneously. 

\textbf{Case 1 :} If $\det(B)=0$, then the single root of the generating polynomial \eqref{eq:YTwo} is
	\begin{equation}
		Y(\omega) = \frac{i\omega \Tr(A)  +\omega^{2} -\det(A)}
		{C-i\omega \Tr(B)},\label{eq:Twovar_Y1}
	\end{equation}
	and the single curve of the ACS is given by 
	\begin{equation}
		\gamma(\omega) = 
		-  \ln
		\left| \frac{i\omega \Tr(A)  +\omega^{2} -\det(A)}
		{C-i\omega \Tr(B)}
		\right|. \label{eq:TwoVar_gamma1}
	\end{equation}
	The value of $\omega_H$ is found from the condition $\gamma(\omega_{H})=0$, which leads to
 \begin{equation}
 \label{eq:omega12}
     \omega_{H}^{4}+\omega_{H}^{2}\left[(\Tr(A))^2-(\Tr(B))^2-2\det(A)\right]+(\det(A))^2-C^{2}=0,
 \end{equation}
 and
 \begin{equation}
    \omega_{H}^{2}=-T_{AB}+\det(A)\pm\sqrt{2\det(A)T_{AB}+T_{AB}^{2}+C^{2}},\quad 
 T_{AB}:=\frac{(\Tr(A))^2-(\Tr(B))^2}{2}. \label{eq:TwoVar_omega}
 \end{equation}
 The phase $\phi_{H}$ is 
 \begin{align}
		\phi_{H} =-\arg\left[\frac{i\omega_H \Tr(A)  +\omega^{2}_H -\det(A)}
		{C-i\omega_H \Tr(B)}\right].\label{eq:TwoVar_phi}
	\end{align}
 The DDE \eqref{eq:TwovarDDE} with $\det (B) = 0$, is of class I if and only if the roots \eqref{eq:TwoVar_omega} are real and simple. It is straightforward to check that the latter is equivalent to the condition $(\det (A))^2 - C^2<0$. Hence, we obtain the following lemma:
	\begin{lem}[Condition for two-variable DDE with $\det B=0$ to be class I]
		\label{lem:3-1}
		The DDE \eqref{eq:TwovarDDE} with $\det (B) = 0$ is of class I if and only if 
  \begin{equation} 
        |\det (A)|< |C|.
  \end{equation}
	\end{lem}

\noindent 
\textbf{Case 2 :} If $\det(B)\neq 0$, the roots of the generating polynomial \eqref{eq:YTwo} are
	\begin{align}
		Y_{1,2}(\omega)=\frac{1}{2\det(B)}\left(i\omega \Tr(B)-C\pm
  \sqrt{z(\omega)}\right),\label{eq:TwoVar_Y2}
	\end{align}
	where  
	\begin{align}
        & z(\omega)=\mu \omega^2+\eta+i\omega \kappa, \label{eq:z}\\
		&\mu=4\det(B)-(\Tr(B))^2,\label{eq:mu} \\ 
        &\eta=C^2-4\det(A)\det(B), \label{eq:eta}\\ 
        &\kappa=4\Tr(A)\det(B)-2\Tr(B)C. \label{eq:kappa}
	\end{align}
 With the notations 
 	\begin{align*}
		K_\pm(\omega)=-C\pm \Re\left(\sqrt{z(\omega)}\right), \quad 
		H_\pm (\omega)=\omega \Tr(B)\pm \Im\left(\sqrt{z(\omega)}\right),
	\end{align*}
the roots can be rewritten as 
	\begin{align}
		Y_{1,2}(\omega)=\frac{1}{2\det(B)}\left[K_\pm(\omega)+iH_\pm(\omega)\right],\label{eq:TwoVar_Y2_2}
	\end{align}
and the ACS has the form	
\begin{equation}
		\label{eq:TwoVar_gamma2}
		\gamma_{1,2}(\omega)=-\frac{1}{2}\ln\left[ \frac{K^2_\pm (\omega)+H^2_\pm (\omega)}{4(\det(B))^2}\right],
\end{equation}

			The value of $\omega_H$  is found from the condition $\gamma_{1}(\omega_{H})=0$ or $\gamma_{2}(\omega_{H})=0$, which is equivalent to  
	\begin{align}
		\label{eq:Two_Var_omegaH2}
		\left[ 
		\frac{K^2_+(\omega_{H})+H^2_+(\omega_{H})}{4(\det(B))^2} - 1
		\right]
		\left[ 
		\frac{K^2_-(\omega_{H})+H^2_-(\omega_{H})}{4(\det(B))^2} - 1
		\right]
		= 0.
	\end{align}
 The following Lemma holds.
	\begin{lem}[Condition for two-variable DDE with $\det (B) \ne 0$ to be class I]
		\label{lem:3.4}
		The DDE \eqref{eq:TwovarDDE} with $\det (B) \ne 0$ is of class I, if and only if equation \eqref{eq:Two_Var_omegaH2} possesses only one pair of simple roots $\omega_H$ and $-\omega_H$. In this case, theorem \ref{thm:1} holds with 
  	\begin{align}
  	    \phi_{H}= -\arg Y_1(\omega_H) =
	\arg\left[\frac{K_{+}(\omega_H)-iH_{+}(\omega_{H})}{2\det(B)}\right]\label{eq:UCI_phi1}
  	\end{align}
	if the simple root $\omega_H$ is the root of the first term in (\ref{eq:Two_Var_omegaH2}), and 
 \begin{align}
     \phi_{H}= -\arg Y_2(\omega_H) =
	\arg\left[\frac{K_{-}(\omega_H)-iH_{-}(\omega_{H})}{2\det(B)}\right] \label{eq:UCI_phi2}
 \end{align}
 otherwise. 
	\end{lem}
 Although Lemma~\ref{lem:3.4} provides an exact criterium, it is not explicitly stated in terms of the coefficients of the DDEs. 
 The following lemma gives simpler \textit{necessary} conditions.
\begin{lem}[Necessary conditions for two-variable DDE to be class I]
		\label{lem:2}
		If DDE \eqref{eq:DDE} with the coefficient matrices \eqref{eq:TwovarDDE} belongs to the universality class I, then the following conditions are satisfied
  \begin{align}
     \det(B)\neq 0,\quad |C| > |\det(A)+\det(B)|,\label{eq:Cnes}
  \end{align} 
where $C$ is defined in \eqref{eq:C}. 
	\end{lem}
	\begin{proof}
 We first note that $\eta>0$ (see Eq.~\eqref{eq:eta}) under the assumptions of the Lemma.
 Indeed, if otherwise, then $z(0)<0$  and $Y_1(0)=Y_2^{*}(0)$ and thus $\gamma_1(0)=\gamma_2(0)$, which contradicts the class I DDE definition.  
 Further, we have
  \begin{equation}
  \label{eq:twovaromega0}
  \gamma_{1,2}(0) = 
  -\frac{1}{2}\ln\left|\frac{1}{4(\det(B))^2}\left(|C| \pm \sqrt{\eta}\right)^2\right|.
  \end{equation}
  For the universality class I, it is necessary that 
\begin{align}
-\frac{1}{2}\ln\left|\frac{1}{4(\det(B))^2}\left(|C| - \sqrt{\eta}\right)^2\right|
> 0 >
-\frac{1}{2}\ln\left|\frac{1}{4(\det(B))^2}\left(|C| + \sqrt{\eta}\right)^2\right|
\end{align}
which is equivalent to 
\begin{align}
\left(|C| - \sqrt{\eta}\right)^2
< 4(\det(B))^2 <
\left(|C| + \sqrt{\eta}\right)^2.
\end{align}
The latter inequality leads to 
  $|C| > |\det(A)+\det(B)|$. Thereby the proof is complete.
	\end{proof}

\subsection{Example for the case \texorpdfstring{$\det(B)=0$}.}  \label{subsec3-3}

To illustrate the two-variable DDEs of universality class I, we consider the following parameter values
\begin{align}
		A=\left(\begin{array}{cc}
			-0.6 & 0.2\\
			0.2& -2
		\end{array}\right),\quad B=\left(\begin{array}{ll}
			1 & -1\\
			-1 & 1
		\end{array}\right).
  \label{eq:AB}
	\end{align}
 As a result, we obtain $\omega_H \approx 1.2893$ from \eqref{eq:TwoVar_omega} and 
$\phi_{H} \approx -0.856$ from \eqref{eq:TwoVar_phi}.
 Then, according to Corollary~\ref{thm:1}, the characteristic equation  has transverse destabilizing crossings with a pair of purely imaginary roots $\pm i\omega_H$ for the following time-delays: 
	\begin{align*}
		\tau_{k}=-0.664+4.871k \quad k=0,1,2,\cdots.
	\end{align*}
 Figure \ref{Fig:TwoVarC0} shows the spectrum for  $\tau_6 \approx 28.6$. The number of unstable eigenvalues for $\tau\in(\tau_{k-1},\tau_k)$ is $2k-1$ for $k=2,3,\cdots$.
 \begin{figure}[H]
		\centering		\includegraphics[width=7.5cm]{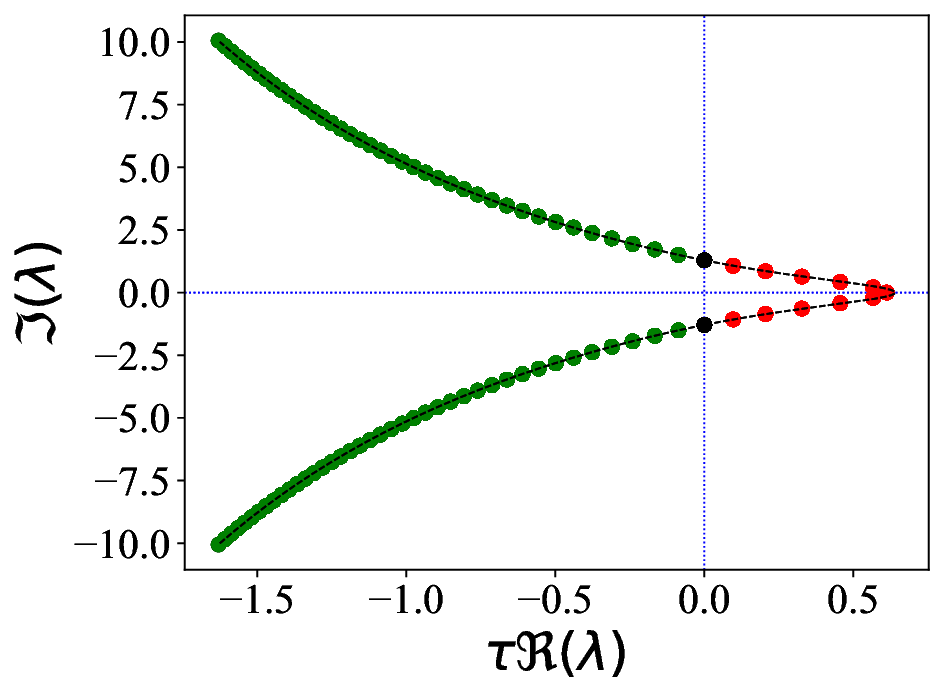}
		\caption{The spectrum of the two-variable delay system with the parameters \eqref{eq:AB} and the delay $\tau_6 \approx 28.6$. The dashed line denotes the curve of the asymptotic continuous spectrum.}
		\label{Fig:TwoVarC0}
	\end{figure}

\subsection{Example for the case \texorpdfstring{$\det(B)\neq 0$}
{}} \label{subsec3-4}

It is instructive to consider two cases: (a)  negative $C$ (see Eq.~\eqref{eq:C}), and (b) positive $C$, both satisfying the necessary condition \eqref{eq:Cnes}. Each of these cases will correspond to the qualitatively different behavior of ACS at $\omega=0$. 
	\begin{align}  
		(a):~
		A=\left(\begin{array}{cc}
			1 & -2\\
			4& -3
		\end{array}\right),\quad B=\left(\begin{array}{cc}
			-3 & 4\\
			-2 & 1.55
		\end{array}\right).\label{eq:ABa}
	\end{align}
 It is straightforward to check that \eqref{eq:Cnes} holds and $C<0$.
 We also obtain $\omega_H\approx 4.088$, $\phi_H\approx -2.977$, leading to the following values of the critical time-delays
\[
\tau_k \approx - 0.733+  1.536 k,\quad k=0,1,2,\cdots.
 \] 
Figure~\ref{Fig:TwoVarC1}(a) shows the corresponding roots of the characteristic equation and the ACS. 
	\begin{align}
		(b):~
		A=\left(\begin{array}{cc}
			2 & 1\\
			3& 1
		\end{array}\right),\quad B=\left(\begin{array}{ll}
			2 & -1\\
			1 & 1
		\end{array}\right).
    \label{eq:ABb}
	\end{align}
	Here we obtain similarly that $C>0$ and \eqref{eq:Cnes} holds, and 
   $\omega_H\approx 1.919$, $\phi_H \approx -1.286$, leading to
\[
\tau_k \approx - 0.670+  3.273 k,\quad k=0,1,2,\cdots.
 \] 
Figure \ref{Fig:TwoVarC1}(b) shows the roots of the characteristic equation. The asymptotic continuous spectra in (a) and (b) differ qualitatively for the unstable part, but this does not change the bifurcation scenarios and the universality class.
	\begin{figure}[H]
		\centering
\subfigure{\includegraphics[width=7.5cm]{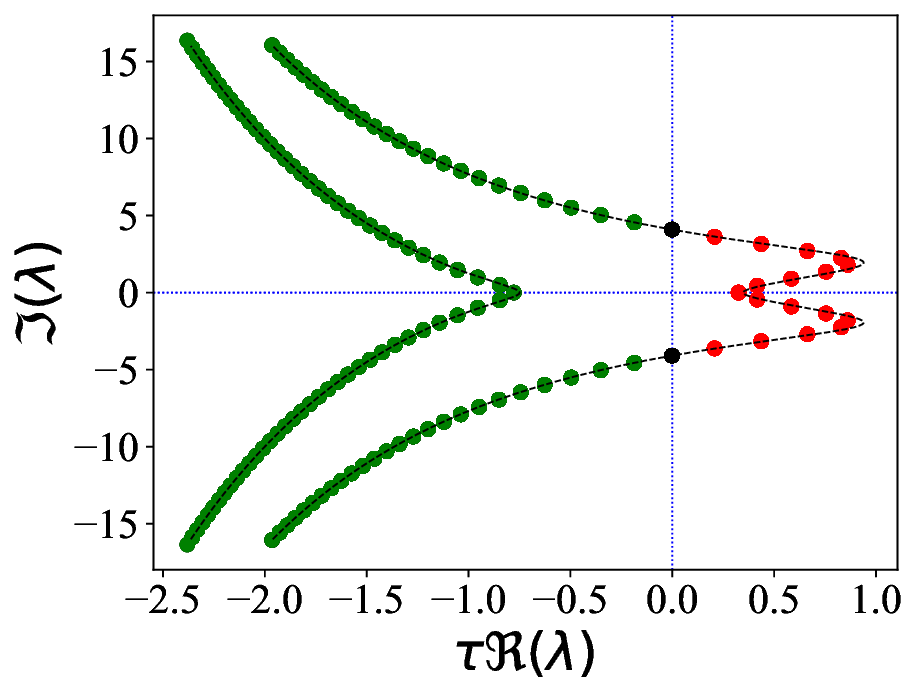}}
\subfigure{\includegraphics[width=7.5cm]{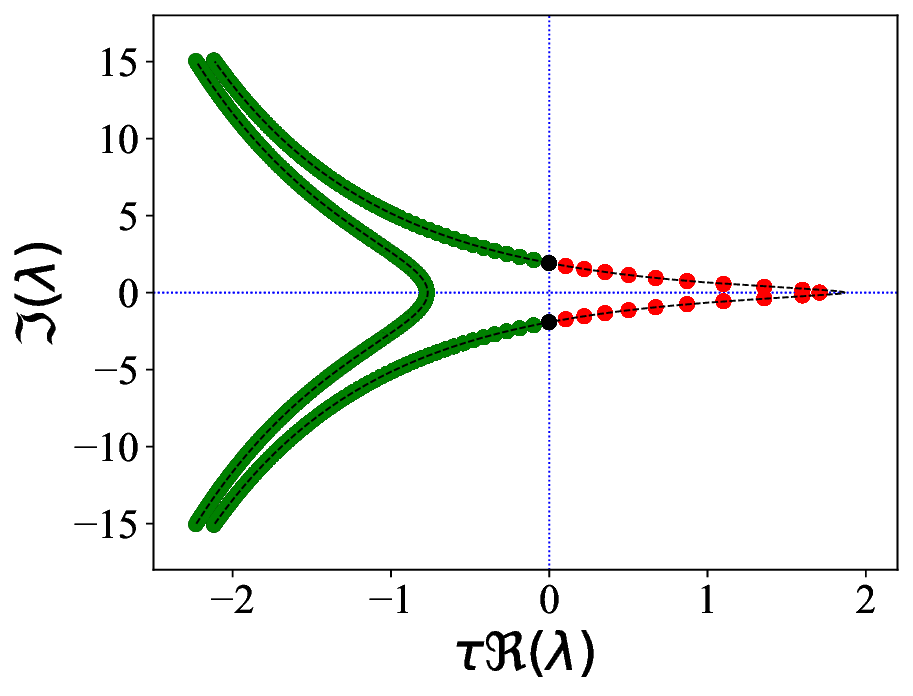}}
		\caption{
  Example of the spectrum for the two-variable delay system \eqref{eq:TwovarDDE} with $\det(B)\ne 0$.
  The coefficients are given in \eqref{eq:ABa} for (a) and \eqref{eq:ABb} for (b). Delay values are (a) $\tau=\tau_9=13.091$, (b) $\tau=\tau_{10}=32.06$. The spectrum in (a) corresponds to a negative value of $C$ (see Eq.~\eqref{eq:C}) and (b) to a positive value of $C$.}
		\label{Fig:TwoVarC1}
	\end{figure}

\section{Universality class II}\label{sec5} 

The universality class II, which appears commonly in applications, corresponds to the ACS shown in Fig.~\ref{fig:classII-1} (a) and (b). This spectrum has two unstable parts $(\omega_{H_2},\omega_{H_1})$ and $(-\omega_{H_1},-\omega_{H_2})$, i.e.
$\gamma_j(\omega_{H_m})=\gamma_h(-\omega_{H_m})=0$ for some $j,h\in\mathbb N$ and $m=1,2$.
Such type of spectrum leads to oscillatory instability and an Eckhaus phenomenon in systems with large time delays \cite{Wolfrum2006,DHuys2008,Yanchuk2015a,Yanchuk2015b,Yanchuk2017}.

	\begin{figure}[H]
		\centering
\subfigure[]{\includegraphics[width=4cm]{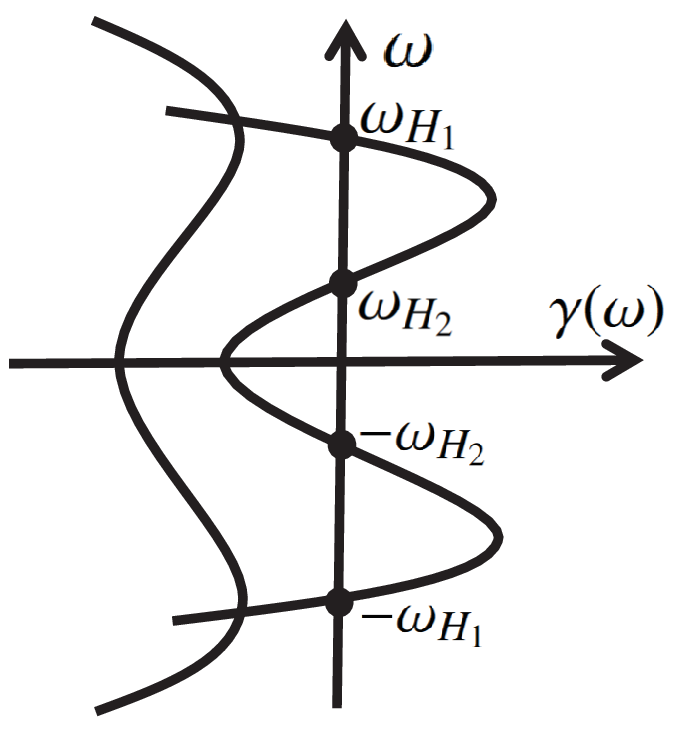}}
\ 
\subfigure[]{\includegraphics[width=4cm]{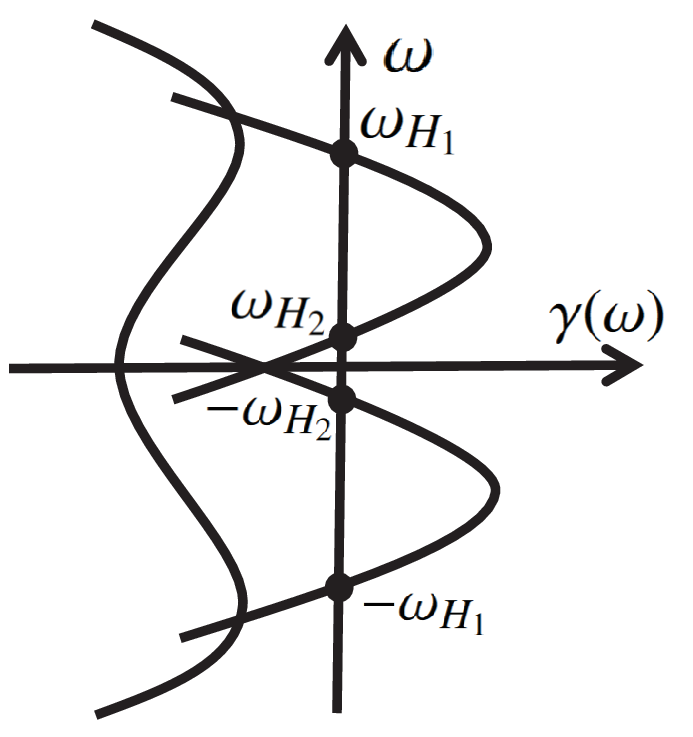}}
		\caption{Schematic structure of universality class II asymptotic spectrum.}
		\label{fig:classII-1}
	\end{figure}

\begin{definition}[Class II ACS] \label{def:4.1} 
		We define the \textit{ACS to be of universality class II} if there exist two positive real numbers $\omega_{H_1}>\omega_{H_2}>0$  such that the branches of the asymptotic continuous spectrum satisfy the following conditions:

  For some $j,h\in \mathbb N$, 
  
  (i) $\gamma_{j}\left(\omega_{H_m}\right)=0~(m=1,2)$,  $\gamma_{j}(\omega)>0$ for all $\omega \in (\omega_{H_2},\omega_{H_1})$.

  (ii) $\gamma_{h}\left(-\omega_{H_m}\right)=0~(m=1,2)$,  $\gamma_{h}(\omega)>0$ for all $\omega \in (-\omega_{H_1},-\omega_{H_2})$.

  (iii) $\omega_{H_m}$ and $-\omega_{H_m}$ are regular points of the functions $\gamma_{j}$ and $\gamma_{h}$, i.e. $\frac{d}{d\omega}\gamma_{h}\left(-\omega_{H_m}\right)$ and $\frac{d}{d\omega}\gamma_{j}\left(\omega_{H_m}\right)$ exist and are nonzero.

  (iv) For all remaining values of $\omega$ and the branches of the ACS not mentioned in (i)-(iii), the ACS is negative. 
	\end{definition}
 Two possible cases of class II spectrum are illustrated in Fig.~\ref{fig:classII-1}: Fig.~\ref{fig:classII-1}(a) for $j=h$ and Fig.~\ref{fig:classII-1}(b) for $j\ne h$.
	\begin{definition}[Class II DDEs] 
 We define the DDEs \eqref{eq:DDE}  to be of \textit{universality class II} if its ACS is of universality class II.
     \end{definition}
	
	We note that Definition \ref{def:4.1} implies that  $\left|Y_{j}(\omega_{H_m})\right| = 1$ and $\left|Y_{h}(-\omega_{H_m})\right| = 1$ $(m=1,2)$. We further define 
	\begin{align}
		\phi_{H_m}=-\arg\left[Y_{j}(\omega_{H_m})\right],\quad m=1,2.
		\label{eq:UC2_phiHm}
	\end{align}
 Theorem \ref{thm:Hopf} implies the following. 
	\begin{cor}[Bifurcation scenario in class II DDEs]
		\label{thm:4.1} Assume  the DDE \eqref{eq:DDE} is of universality class II. 
		Then system \eqref{eq:DDE} possesses delay-induced transverse crossings of the characteristic roots at $\lambda_c=\pm i\omega_{H_m}$ for the following values of time-delay 
		\begin{align}
			\tau_k^{(m)}=\frac{1}{\omega_{H_m}}\left(\phi_{H_m}+2\pi k\right),\quad k=0,1,2,\cdots,\label{eq:UC2_tau}
		\end{align}
   where $m=1,2$ and  $\phi_{H_m}$ are defined by Eq.~\eqref{eq:UC2_phiHm}.
		Moreover, there are no other delay-induced transverse crossings in the system. The crossings for  $m=1$ are destabilizing and for $m=2$ stabilizing, i.e.
       \begin{align*}
          \left.\frac{d}{d\tau}\Re[\lambda(\tau)]\right|_{\lambda=\pm i\omega_{H_1},\tau=\tau_{k}^{(1)}}>0,\quad \left.\frac{d}{d\tau}\Re[\lambda(\tau)]\right|_{\lambda=\pm i\omega_{H_2},\tau=\tau_{k}^{(2)}}<0.
      \end{align*}
	\end{cor}

In contrast to class I systems, class II systems can have two pairs of critical characteristic roots, which may lead to double-Hopf bifurcations in nonlinear DDEs. 
The condition for this is $\tau_k^{(1)}=\tau_l^{(2)}$ with some $k,l\in\mathbb N$.
The following theorem provides a countable number of conditions for the appearance of two pairs of critical roots. 

\begin{prop}[
\textbf{Two pairs of critical characteristic roots in class II DDEs}
]
\label{thm:double-classII}
    Assume that DDE \eqref{eq:DDE} is of universality class II. 
	Then the characteristic equation \eqref{eq:CharEq} possesses two pairs of purely imaginary roots $\pm i\omega_{H_m},~m=1,2$ if and only if 
 \begin{equation}
     	\frac{1}{\omega_{H_1}}\left(\phi_{H_1}+2\pi k\right)
      = 
      \frac{1}{\omega_{H_2}}\left(\phi_{H_2}+2\pi l\right),
      \quad k,l \in \mathbb N.
 \end{equation}
\end{prop}

\subsection{First example for universality class II} \label{subsec4-1}

	We consider DDE \eqref{eq:DDE} with the following matrices
	\begin{align}\label{ExUC2_1}
		A=\left(\begin{array}{cc}
			-\alpha & \beta\\
			-\varphi& 0
		\end{array}\right),\quad B=\left(\begin{array}{ll}
			-\alpha & -\beta\\
			0 & -\varphi
		\end{array}\right),
	\end{align}
	where $\alpha=0.5$, $\beta=4.5$, and $\varphi=1.5$. 
 The ACS is given by Eq.~\eqref{eq:TwoVar_gamma2} with $\det (B)=0.75$, $\Tr(B)= -2$,  $C=-6$, $z(\omega) = -\omega^2 + 15.75 - i 25.5 \omega$. We also obtain $\phi_{H_1} = 2.481 $, $\phi_{H_2} =-0.492$, $\omega_{H_1} =  4.085 $, $\omega_{H_2} =0.977 $, and  the sequence of critical time-delays  are 
\begin{align*}
    \tau_k^{(1)} =  0.607+  1.537 k,\quad \tau_k^{(2)} = -0.5036+6.4278 k,\quad k=0,1,2,\cdots.
\end{align*}

Figure \ref{Fig:UC2_1} shows the roots of the characteristic equation and the ACS for $\tau=\tau_{12}^{(1)}=19.051$. 
  
	\begin{figure}[H]
		\centering		\subfigure{\includegraphics[width=7.5cm]{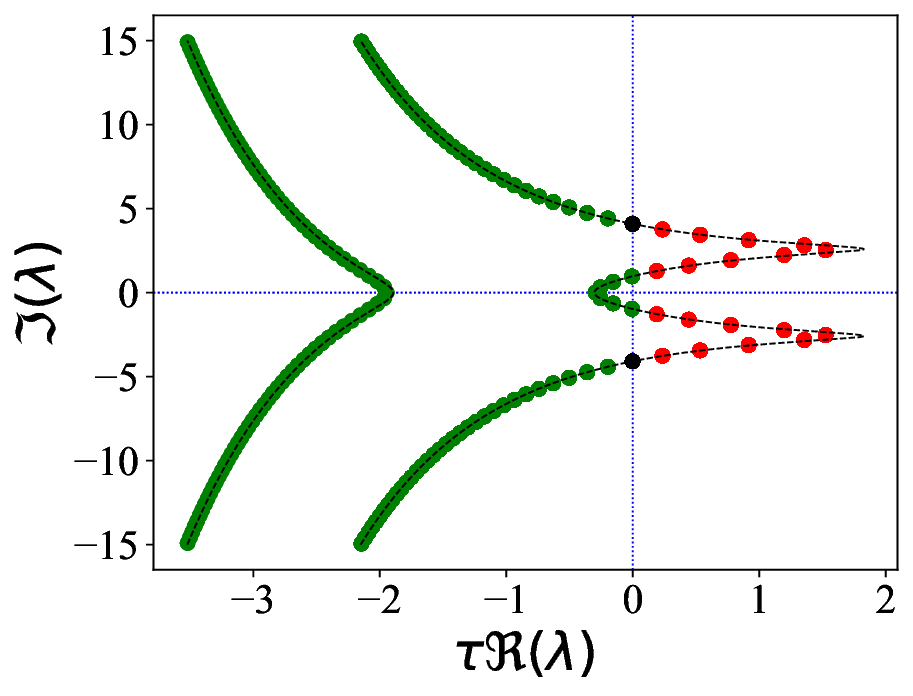}}
		\caption{Example of the spectrum for the two-variable delay system \eqref{ExUC2_1} with  $\tau= 19.051$ and $\alpha=0.5$, $\beta=4.5$, and $\varphi=1.5$.}
		\label{Fig:UC2_1}
	\end{figure}
 
\subsection{Second example for universality class II }\label{subsec4-2}
 
Here we present an example with the spectrum shown in Fig.~\ref{fig:classII-1}(b). 
For this we consider system \eqref{eq:DDE} with the parameters
	\begin{align}\label{ExUC2_2}
		A=\left(\begin{array}{cc}
			\alpha & \beta\\
			-\beta& \alpha
		\end{array}\right),\quad B=\left(\begin{array}{ll}
			\mu & 0\\
			0 & \mu
		\end{array}\right),
	\end{align}
	where $\alpha,~\beta$, and $\mu$ are  real parameters with $\beta>0$ and $\mu\neq 0$. 
 A similar system appears via a linearization in \cite{Wolfrum2006}. The corresponding 
  characteristic equation is 
	\begin{align}
		(\lambda-\alpha-\mu e^{-\lambda \tau})^2+\beta^2=
        \left( 
        \lambda-\alpha-\mu e^{-\lambda \tau} + i\beta 
        \right)
        \left( 
        \lambda-\alpha-\mu e^{-\lambda \tau} - i\beta 
        \right)
        = 0,
  \label{eq:CqUc2}  
	\end{align}
    and the generating polynomial 
	has two roots 
	\begin{align*}
		Y_{1,2}=-\frac{\alpha}{\mu}+i\left(\frac{\omega\pm \beta}{\mu}\right).
	\end{align*}
Hence the ACS is composed of two curves given by  
	\begin{align}
		\gamma_{1,2}(\omega)=-\frac{1}{2}
        \ln\left[
            \alpha^2 + (\omega\pm\beta)^2 
            \right] 
        + \ln |\mu|.
  \label{eq:ac-Uc2}
	\end{align}
 The following condition for the system to be of universality class II is straightforward.
\begin{prop}[Condition for system \eqref{eq:DDE}, \eqref{ExUC2_2} to be of class II]
		\label{lem:4.2}
 System \eqref{eq:DDE} with matrices \eqref{ExUC2_2} belongs to the universality class II if and only if the following conditions are satisfied
  \begin{align}
     |\alpha | < |\mu|,\quad \alpha^2+\beta^2>\mu^2.
  \end{align} 
	\end{prop}
	\begin{proof}
 Firstly, we have
  \begin{equation}
  \gamma_{1,2}(0) =-\frac{1}{2}\ln\left[\frac{\alpha^2}{\mu^2}+\frac{\beta^2}{\mu^2} \right].
  \end{equation}
  From \eqref{eq:ac-Uc2}, we obtain
	\begin{align}
		\omega_{H_1}=\beta+\sqrt{\mu^2-\alpha^2}, \quad \omega_{H_2}=\beta-\sqrt{\mu^2-\alpha^2}.\label{eq:omega_Hclass2}
	\end{align}
  Then, for the system to belong to the universality class II, it is necessary that $\gamma_{1}(0)=\gamma_{2}(0)=\gamma(0)<0$, which leads to  $\alpha^2+\beta^2>\mu^2$.
  Finally note that the condition $|\alpha|<|\mu|$ is necessary and sufficient for the existence of two real positive roots $\omega_{H_1}>\omega_{H_2}$.  
 \end{proof}
	
The phases $\phi_{H_m}$ are given as
	\begin{align}
 \label{eq:phi_Hclass2_1}
		\phi_{H_1}&=-\arg Y_1(\omega_{H_1})=\arg\left[-\frac{\alpha}{\mu}-i\frac{2\beta+\sqrt{\mu^2-\alpha^2}}{\mu} \right],\\ 
		\phi_{H_2}&=-\arg
		Y_1(\omega_{H_2})=\arg\left[-\frac{\alpha}{\mu}-i\frac{2\beta-\sqrt{\mu^2-\alpha^2}}{\mu} \right].\label{eq:phi_Hclass2_2}
	\end{align}
	Corollary~\ref{thm:4.1} implies that transverse crossings with  $\lambda_c = \pm i \omega_{H_m}$ occur for the following time-delays:
	\begin{align}\label{eq:tau1Exam}
		\tau^{(1)}_k&=\frac{1}{\omega_{H_1}}\left(\phi_{H_1}+2\pi k\right)=\frac{1}{\beta+\sqrt{\mu^2-\alpha^2}}\arg\left[-\frac{\alpha}{\mu}-i\frac{2\beta+\sqrt{\mu^2-\alpha^2}}{\mu} +2\pi k \right],\\ \label{eq:tau2Exam}
		\tau^{(2)}_{k}&=\frac{1}{\omega_{H_2}}\left(\phi_{H_2}+2\pi k\right)=\frac{1}{\beta-\sqrt{\mu^2-\alpha^2}}\arg\left[-\frac{\alpha}{\mu}-i\frac{2\beta-\sqrt{\mu^2-\alpha^2}}{\mu} +2\pi k \right], \\ &k=0,1,2,\cdots.\nonumber
	\end{align}
	
 Let us illustrate the spectrum for specific parameter values $\alpha=0.5$, $\beta=3.5$, $\mu=2.2$. 
 Using Eqs.~\eqref{eq:omega_Hclass2}--\eqref{eq:phi_Hclass2_2}, we have $\phi_{H_1} = -1.6254 $, $\phi_{H_2} =-1.6734$, $\omega_{H_1} =  5.6424 $, $\omega_{H_2} =1.3576 $ and the sequence of critical characteristic roots appear for 
\[
\tau_k^{(1)} = -0.288+  1.113 k,\quad
 \tau_k^{(2)} = 1.233+4.626k,\quad k=0,1,2,\cdots.
 \]
  Figure \ref{Fig:UC2_2} shows the roots of the corresponding characteristic equation and the ACS for $\tau=\tau_{5}^{(2)}=24.363$.
	\begin{figure}[H]
		\centering	\includegraphics[width=7.5cm]{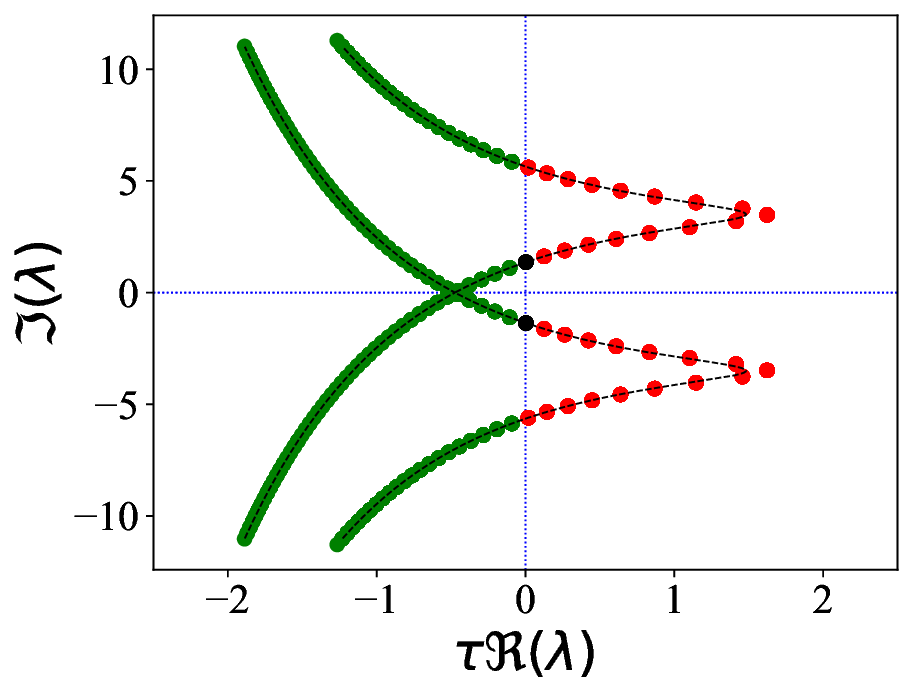}
		\caption{Spectrum of the two-variable delay system \eqref{eq:DDE}, \eqref{ExUC2_2} with  $\tau=\tau_{5}^{(2)}=24.363$,  $\alpha=0.5$, $\beta=3.5$, and $\mu=2.2$. }
		\label{Fig:UC2_2}
	\end{figure}

\section{Universality class III}\label{sec6} 

The universality class III, which appears commonly in applications, corresponds to the asymptotic spectrum shown in Fig.~\ref{fig:UC3} with two unstable parts $(-\omega_{H},0) \bigcup (0, \omega_{H})$, $\gamma_j(0)=0$, $\gamma_j(\pm \omega_H)=0$, $j\in \mathbb{N}$.
Such spectrum is common for systems with symmetries such as the Stuart-Landau system with time delay \cite{Wolfrum2006,Fukuda2007,Perlikowski2010a,Selivanov2012,Yanchuk2015a,Martin2016,Zhang2018} or the Lang-Kobayashi laser model \cite{Lang1980, Alsing1996, Heil1998, Mulet1999, Pieroux2003, Green2009,  Yanchuk2010a, Dahms2010,  Wolfrum2010, RoehmBoehmLuedge2016}.
 \begin{figure}[H]
		\centering 
\subfigure{\includegraphics[width=4cm]{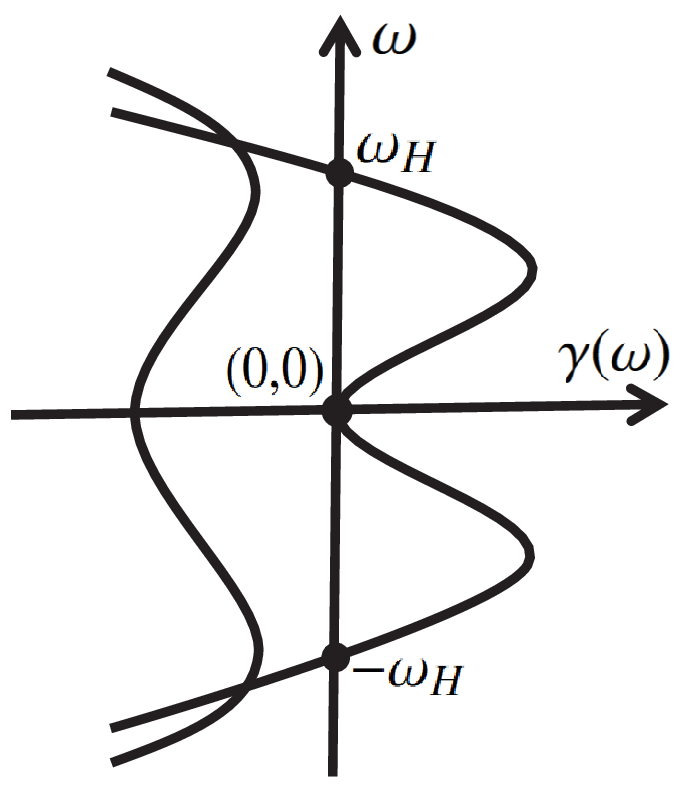}}
		\caption{Schematic structure of the class III asymptotic spectrum.}
		\label{fig:UC3}
	\end{figure}
\begin{definition}[Class III ACS] \label{as:uc3} 
    \textit{The ACS is of universality class III} if there exist $\omega_{H}>0$ and $j\in \mathbb{N}$ such that a branch of the ACS satisfies  $\gamma_{j}\left(0\right)=0$ and $\gamma_{j}\left(\pm\omega_{H}\right)=0$,  $\gamma_{j}(\omega)>0$ for all $\omega \in (-\omega_{H},0)\bigcup (0, \omega_H)$, and $\gamma_{j}(\omega)<0$ for all $\omega \not\in [-\omega_{H},\omega_H]$.
    Moreover, $\gamma_{k}\left(\omega\right)\ne 0$
    for all $k\ne j$ and $\omega\in\mathbb{R}$. 
\end{definition}

\begin{definition}[Class III DDE]
    DDEs \eqref{eq:DDE} is of \textit{universality class III} if its ACS is of universality class III. 
\end{definition}

As in the previous cases, Definition \ref{as:uc3} implies that  $\left|Y_{j}(\omega_{H})\right| = 1$ and we define 
\begin{align}
\phi_{H}=-\arg\left[Y_{j}(\omega_{H})\right].
\label{eq:UC3phi_H}
\end{align}

The following corollary of the Theorem~\ref{thm:Hopf} for the class III DDEs describes the sequences of transverse crossings (bifurcations) in such DDEs. 

\begin{cor}[Bifurcation scenario in class III DDEs]
\label{thm:5_1} Assume  that the DDE \eqref{eq:DDE} is of universality class III. 
Then the characteristic equation \eqref{eq:CharEq} has a characteristic root  $\lambda_0=0$ for all time-delays if and only if $\det(A+B)=0$.  Also, the system possesses destabilizing transverse crossings of the characteristic roots with  $\lambda_c=\pm i\omega_{H}$ for the following time-delays 
\begin{align}
\tau_{k}=\frac{1}{\omega_{H}}\left(\phi_{H}+2\pi k\right),\quad k=0,1,2,\cdots,\label{eq:UC3_tauHopf}
\end{align}
where $\phi_{H}$ is defined by Eq.~\eqref{eq:UC3phi_H}.
Moreover, there are no other delay-induced transverse crossings of the characteristic roots in \eqref{eq:DDE}. 
\end{cor}
\begin{proof}
The root $\lambda_0=0$ is only possible if 
\begin{align}
\det\left[A+B\right]=0.\label{eq:YHopf3_1}
\end{align}
This is the case when $Y_{j}(0)=1$ or, equivalently,  $\varphi_H=0$. 
The remaining statement follows from Theorem~\ref{thm:Hopf}. 
\end{proof}

\subsection{Two-variable example of DDE of universality class III with \texorpdfstring{$\det (B) = 0$}. }
 
We remind that there exists only one curve of the ACS for $\det (B) = 0$ in the case of two-variable DDE. Using Eq.~\eqref{eq:Twovar_Y1} from section \ref{subsec3-2}, the single root of the generating polynomial  at $\omega=0$ is 
\begin{equation}
		Y(0) = -\frac{ \det(A)}
		{C},\label{eq:Y0}
	\end{equation}
 and the asymptotic continuous spectrum at $\omega=0$ is 
\begin{equation}
		\gamma(0) = 
		-  \ln
		\left| \frac{ \det(A)}
    		{C}
		\right|. \label{eq:gamma0}
	\end{equation}
 Then, the condition $\gamma(0)=0$ leads to $|C|=|\det(A)|$.  
 Further, the condition $\gamma(\omega_H)=0$, $\omega_H\ne 0$ is equivalent to
 \begin{align}
     \omega^2_H+(\Tr(A))^2-(\Tr(B))^2-2\det(A)=0
     \label{eq:omega_0}
 \end{align}
 as it can be seen from Eq.~\eqref{eq:omega12}. 
Then, the DDE \eqref{eq:DDE} with \eqref{eq:TwovarDDE} and $\det (B) = 0$ is of class III if and only if equation \eqref{eq:omega_0} possesses a pair of simple roots $\omega_H$ and $-\omega_H$. The latter is equivalent to the condition $(\Tr(A))^2-(\Tr(B))^2<2\det(A) $. Hence, we obtain the following result.
	\begin{lem}
		\label{lem:5-1}
		The tw-variable DDE \eqref{eq:DDE} with \eqref{eq:TwovarDDE} and $\det (B) = 0$ is of class III if and only if 
  \begin{equation} 
        |\det (A)| =|C|, \quad (\Tr(A))^2-(\Tr(B))^2<2\det(A).
  \end{equation}
	\end{lem}

As a numerical example of universality class III with $\det(B)= 0$ 
we consider
\begin{align}
		A=\left(\begin{array}{cc}
			1 & -2\\
			4 & -3
		\end{array}\right),\quad B=\left(\begin{array}{ll}
			2 & -1\\
			-2 & 1
		\end{array}\right),
  \label{eq:ExUC3_1}
	\end{align}
 for which we have $\omega_H= 3.873$ and $\phi_{H}=-1.3181$. Then the destabilizing transverse crossings of the critical roots occur for the following time delays: 
	\begin{align*}
		\tau_{k}=-0.34+1.621 k, \quad k=0,1,2,\dots.
	\end{align*}
 Figure \ref{Fig:UC3_1} shows the roots of the characteristic equation for  $\tau_{10}=15.87$ together with the ACS.
 \begin{figure}[H]
		\centering
		\includegraphics[width=7.5cm]{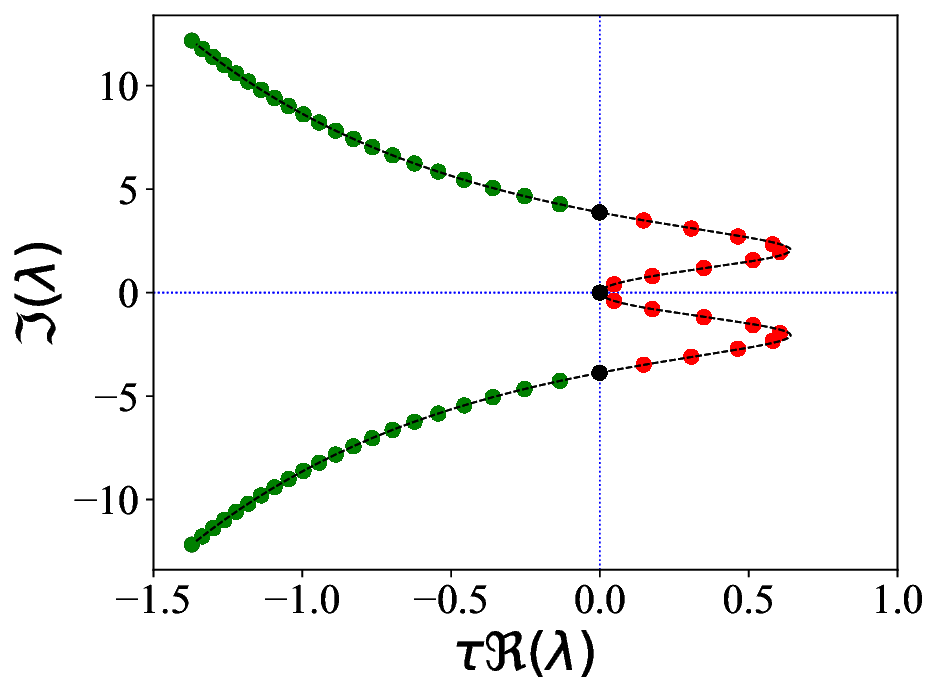}
		\caption{The spectrum of the two-variable delay system \eqref{eq:DDE}--\eqref{eq:TwovarDDE} with the coefficients \eqref{eq:ExUC3_1} belongs to universality class III. An example of the spectrum for $\tau=\tau_{10}=15.87$ and the ACS are shown.}
		\label{Fig:UC3_1}
	\end{figure}

\subsection{Two-variable example of DDE of universality class III with \texorpdfstring{$\det (B) \ne 0$}. }

We obtain the following two necessary conditions for a DDE to be of universality class III with $\det(A+B)=0$ and $\det(A-B)=0$ respectively. Although it looks like these conditions are also sufficient, we were not able to prove this. 
 \begin{lem}\label{lem:5-2}
    Let DDE \eqref{eq:TwovarDDE} with $\det (B) \ne 0$ and $\det(A+B)=0$.    
    belongs to the universality class III. Then the following conditions hold 
    \begin{align}\label{eq:condition1}
        &|\det(A)|>|\det(B)|,\\
                \label{eq:condition2}
        &2\left[\nu_2-\nu_1\Tr(B)+\left(\frac{\nu_1}{\nu_2}\right)^2\det(B)\right]+\nu^2_1<0,
    \end{align}
   where 
   \begin{align}
       \nu_1=\Tr(A)+\Tr(B), \quad \nu_2=\det(B)-\det(A).
    \end{align}
 \end{lem}
 \begin{proof}
     From Eq.~\eqref{eq:TwoVar_Y2}, the roots of the generating polynomial \eqref{eq:YTwo} at $\omega=0$ are
	\begin{align}
		Y_{1,2}(0)=\frac{1}{2\det(B)}\left(-C\pm
  \sqrt{\eta}\right),\label{eq:UC3_Y2_0}
	\end{align}
 and  the values of the asymptotic continuous spectrum at $\omega=0$ are 
\begin{equation}
    \gamma_{1,2}(0) = 
		-\frac{1}{2}\ln\left|\frac{1}{4(\det(B))^2}\left(|C| \pm\sqrt{\eta}\right)^2\right|. \label{eq:UC3_gamma2_0}
	\end{equation}
Since $p_{0}(1)=0$ is equivalent to $\det(A+B)=0$, where  $p_{\omega}(Y)$ is the generating polynomial \eqref{eq:Y}, the condition  $\det(A+B)=0$ implies $Y_2(0)=1$ and $\gamma_2(0)=0$. 
The necessary condition for the universality class III is that the ACS of DDE \eqref{eq:TwovarDDE} consists of two curves such that $\gamma_1(0) <\gamma_2(0)=0$. Hence, we obtain $\eta > 0$,
$Y_2(0)=1$, $|Y_1(0)|>1$, and
\begin{align}
\left(|C| - \sqrt{\eta}\right)^2
= {4(\det(B))^2} <
\left(|C| + \sqrt{\eta}\right)^2.
 \label{eq:suffCond0}
\end{align}
Using the expressions for $C$ and $\eta$, as well as the condition $\det (A+B)=0$, the inequality \eqref{eq:suffCond0} is reduced to $|\det(A)|>|\det(B)|$.

Next, we compute the second derivative of the branch of ACS $\gamma(\omega)$ at $\omega=0$. For this we differentiate  \eqref{eq:YTwo} with respect to $\omega$ as
\begin{align}
    -2\omega-i\left[\Tr(A)+Y_2\Tr(B)\right]-i\omega  \Tr (B)
 \partial_{\omega}Y_2+2Y_2\det(B)\partial_{\omega}Y_2+C\partial_{\omega}Y_2 =0.
    \label{eq:oneD}
\end{align}
Substituting $Y_2 = 1$, $\omega = 0$ and $\det(A+B)=0$ in \eqref{eq:oneD}, we obtain
\begin{align}
\partial_{\omega}Y_2(0)=i\frac{\Tr(A)+\Tr(B)}{\det(B)-\det(A)}=i\frac{\nu_1}{\nu_2}.
\label{eq:oneD_0}
\end{align}
The second derivative of \eqref{eq:YTwo} with respect to $\omega$ is
\begin{align}
    -2-2i\Tr(B)\partial_{\omega}Y_2-i\omega\Tr(B)\partial_{\omega\omega}Y_2+2\left[Y_2\partial_{\omega\omega} Y_2+(\partial_{\omega}Y_2)^2\right]\det(B)+C\partial_{\omega\omega}Y_2=0,
    \label{eq:twoD}
\end{align}
which leads for $Y_2= 1$, $\omega = 0$, and $\det(A+B)=0$ to 
\begin{align}
\partial_{\omega\omega}Y_2(0)=\frac{2}{\nu_2}\left[1-\frac{\nu_1}{\nu_2}\Tr(B)+\left(\frac{\nu_1}{\nu_2}\right)^2\det(B)\right].
\label{eq:twoD_0}
\end{align}
Then, the second derivative of $\gamma$ is
\begin{align*}
  \frac{d^2\gamma_2}{d\omega^2}(0)=\left|\partial_{\omega}Y_2(0)\right|^2-\left|\partial_{\omega\omega}Y_2(0)\right|.
\end{align*}

For the universality class III, it is necessary that $\gamma_2''(0) >0$ leading to 
$\left|\partial_{\omega}Y_2(0)\right|^2>\partial_{\omega\omega}Y_2(0)$. The latter inequality, combined with \eqref{eq:suffCond0}, \eqref{eq:oneD_0}, and \eqref{eq:twoD_0}, give the conditions of the lemma.
\end{proof}

\begin{rem}
The conditions of Lemma \ref{lem:5-2} are reduced to $\det(A)>\det(B)$ in the case of $\partial_{\omega}Y_2(0)=0$ (i.e. $\nu_1=0$).
\end{rem}
 
We give a numerical example of universality class III with $\det(B)\ne 0$, $\det(A+B)= 0$, and 
 \begin{align}
		A=\left(\begin{array}{cc}
			1 & -2\\
			4& -2.2
		\end{array}\right),\quad B=\left(\begin{array}{ll}
			-1 & -0.2\\
			-4 & 3
		\end{array}\right).
  \label{eq:EXUC3_2}
	\end{align}
We obtain  $\det(A)=5.8$, $\det(B)=-3.8$, $\Tr(A)=-1.2$, $\Tr(B)=2$, $\nu_1=0.8$, and $\nu_2=-9.6$. it is easy to verify that conditions  \eqref{eq:condition1} and \eqref{eq:condition2} are satisfied. 
Further, we calculate $\phi_{H} = -1.379$, $\omega_{H} =  4.13$, and the sequence of the delay-induced transverse crossings of the characteristic roots occur at the time-delays
\[
\tau_k = -0.334+ 1.521k,\quad k=0,1,2,\cdots.
 \]
 Figure \ref{fig:UC3_2} shows the roots of the characteristic equation of the above two-variable linear system and its spectrum curves for $\tau=\tau_{20}=30.086$.
  \begin{figure}[H]
		\centering
\subfigure{\includegraphics[width=7.5cm]{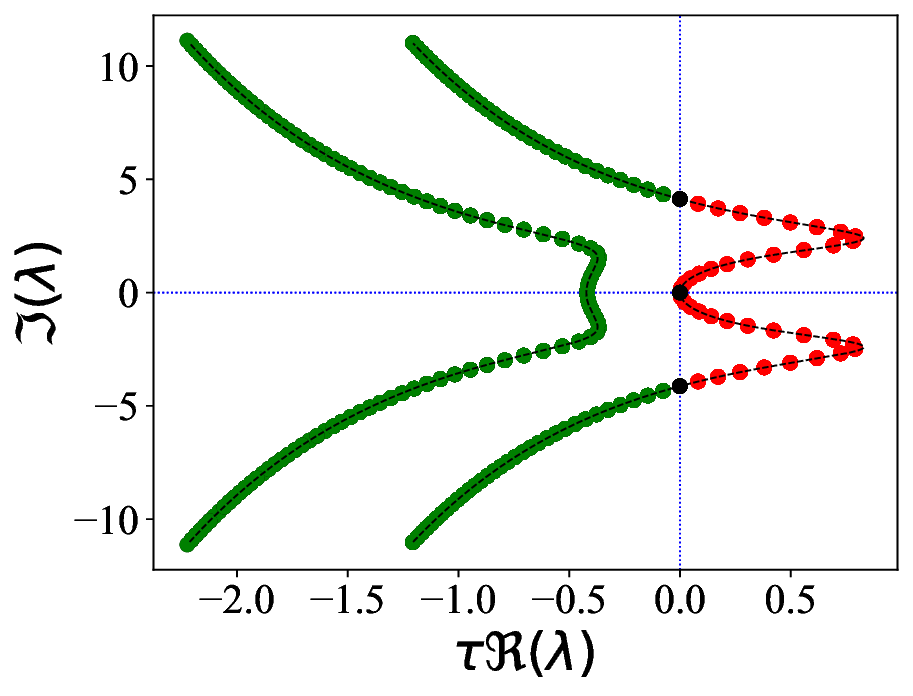}}
		\caption{
  Example of the spectrum for two-variable delay system \eqref{eq:TwovarDDE}.
  The coefficients are given in \eqref{eq:EXUC3_2} and the time-delay value is $\tau=\tau_{20}=30.086$.}
		\label{fig:UC3_2}
	\end{figure}

The second necessary condition is given by the following result.
\begin{lem}\label{lem:5-3}
    Let the DDE \eqref{eq:TwovarDDE} with $\det (B) \ne 0$ and $\det(A-B)=0$ belongs to universality class III. Then the following conditions hold
     \begin{align}\label{eq:condition3}
       & C^2>\max\big\{4\det(A)\det(B), 2\det(A)\det(B)+2(\det(B))^2\big\}, \\ 
               \label{eq:condition4}
        &2\left[\nu_4+\nu_3\Tr(B)+\left(\frac{\nu_3}{\nu_4}\right)^2\det(B)\right]-\nu^2_3 > 0, 
    \end{align}
    where
\begin{align}
\label{eq:nus}
\nu_3 =  \Tr(A)-\Tr(B), \nu_4 =  \det(A+B)  - \det(A) - 3\det(B).
\end{align}
\end{lem}
 \begin{proof}
 
The necessary condition for the universality class III is that the ACS of DDE \eqref{eq:DDE}, \eqref{eq:TwovarDDE} consists of two curves such that $\gamma_1(0) <\gamma_2(0)=0$. Hence we have  $\eta > 0$ (see eq.~\eqref{eq:twovaromega0}). Also $\det(A-B)=0$ implies $Y_2(0)=-1$, since $p_0(-1)=\det(A-B)=0$. 
The condition $\gamma_1(0) <0$ implies $|Y_1(0)|>1$, and so, similar to the proof of Lemma~\ref{lem:5-2}, we can show that 
\eqref{eq:suffCond0} holds. 
Simple calculations show that \eqref{eq:suffCond0} is equivalent to $C^2>2\det(A)\det(B)+2(\det(B))^2$. 
Furthermore, $\eta > 0$ is equivalent to $C^2>4\det(A)\det(B)$. Hence, we have shown the condition \eqref{eq:condition3}.

Next, substituting $Y_2 = -1$, $\omega = 0$ into \eqref{eq:oneD}, we obtain
\begin{align}
\partial_{\omega}Y_2(0)=i\frac{\Tr(A)-\Tr(B)}{\det(A+B)-\det(A)-3\det(B)} =
i\frac{\nu_3}{\nu_4}.
\label{eq:oneD_1}
\end{align}
Further, using \eqref{eq:twoD} with $Y_2= -1$ and $\omega = 0$, we obtain
\begin{align}
\partial_{\omega\omega}Y_2(0)=\frac{2}{\nu_4}\left[1 +\frac{\nu_3}{\nu_4}\Tr(B)+\left(\frac{\nu_3}{\nu_4}\right)^2\det(B)\right].
\label{eq:twoD_1}
\end{align}
Similar to the proof of Lemma \ref{lem:5-2}, we obtain
\begin{align*}
\frac{d^2\gamma_2}{d\omega^2}(0)=\left|\partial_{\omega}Y_2(0)\right|^2+\left|\partial_{\omega\omega}Y_2(0)\right|.
\end{align*}
Since the obtained second derivative must be positive, we have $\left|\partial_{\omega}Y_2(0)\right|^2+\partial_{\omega\omega}Y_2(0)>0$. Then, combining  \eqref{eq:suffCond0}, \eqref{eq:oneD_1}, and \eqref{eq:twoD_1},  straightforward calculations lead to the conditions \eqref{eq:condition3}--\eqref{eq:condition4}.
\end{proof}
\begin{rem}
Conditions of Lemma \ref{lem:5-3} are reduced to $C^2>\max\{4\det(A)\det(B), 2\det(A)\det(B)+2(\det(B))^2\}$ and $\nu_4>0$ in the case $\partial_{\omega}Y_2(0)=0$ (i.e. $\nu_3=0$).
\end{rem}
  We give here a numerical example of universality class III with $\det(B)\ne 0$, $\det(A-B)= 0$, and 
 \begin{align}
		A=\left(\begin{array}{cc}
			1 & -2\\
			4& -3
		\end{array}\right),\quad B=\left(\begin{array}{ll}
			1 & -0.5\\
			4 & -3
		\end{array}\right).
  \label{eq:EXUC3_3}
	\end{align}
It is easy to check that conditions of Lemma~\ref{lem:5-3} are satisfied, and ACS belongs to the universality class III with the critical delays
\[
\tau_k = 0.584+ 1.924k,\quad k=0,1,2,\cdots.
 \]
 Figure \ref{fig:UC3_3} shows the roots of the characteristic equation of the above two-variable linear system and its ACS.
  \begin{figure}[H]
		\centering
\subfigure{\includegraphics[width=7.5cm]{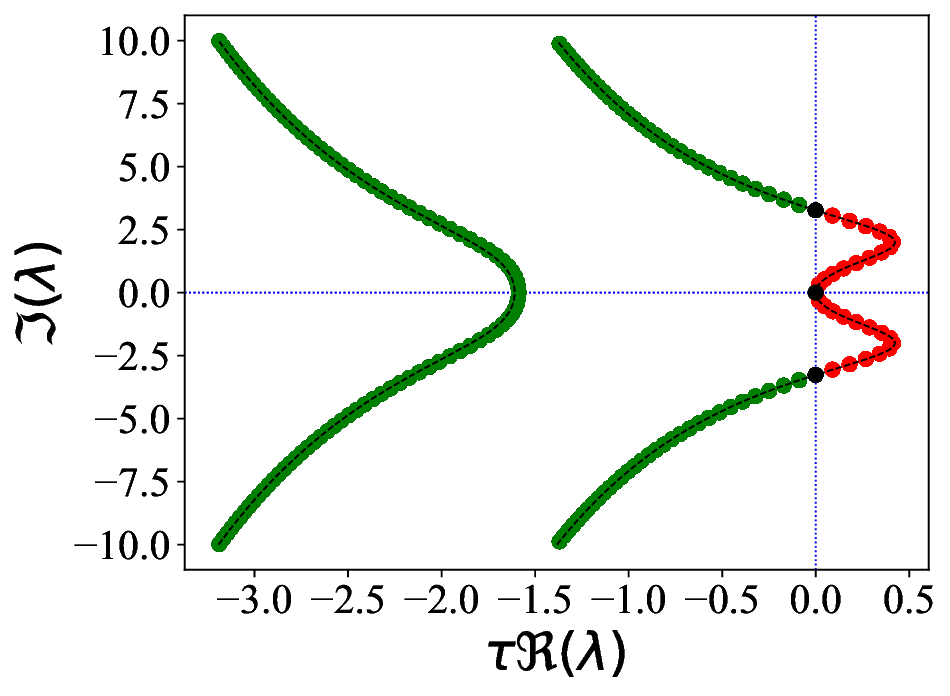}}
		\caption{
  Example of the spectrum for two-variable delay system \eqref{eq:TwovarDDE}.
  The coefficients are given in \eqref{eq:EXUC3_3} and the time-delay  is $\tau=\tau_{15}=29.44$.}
		\label{fig:UC3_3}
	\end{figure}
 
\section{Nonlinear model example: Stuart-Landau oscillator with time-delayed feedback}\label{sec7} 
\subsection{Asymptotic continuous spectrum analysis}\label{subsec6-1} 
In this section we consider the Stuart-Landau model  \cite{ Wolfrum2006, Fukuda2007, Fiedler2007, Yanchuk2009, Perlikowski2010a, Selivanov2012, zakharovaTimeDelayControl2013, Martin2016} with time-delayed feedback. 
This is a paradigmatic model for studying
the interplay between oscillatory instability (Hopf bifurcation) with a time delay. 
This example also illustrates how the Hopf bifurcation scenario in a DDEs of universality class II follows from our results.
The system reads
\begin{align}
    \dot z(t)=(\alpha+i\beta)z(t)-z(t)|z(t)|^2+z(t-\tau),
    \label{eq:SL}
\end{align}
where $\alpha$ and $\beta$ are real parameters, $z(t):\mathbb{R}\mapsto\mathbb{C}$ is a complex variable. The system has two real variables: the real and the imaginary parts of $z$.

We start with the stability analysis of the trivial equilibrium $z=0$ by considering the linearized system 
\begin{align}
    \dot z(t)=(\alpha+i\beta)z(t)+z(t-\tau).
    \label{eq:linearSL}
\end{align}
The corresponding characteristic equation is
\begin{align}
    \lambda-(\alpha+i\beta)-e^{-\lambda\tau}=0, \label{eq:CharaE}
\end{align}
and the generating polynomial 
\begin{align}
p_\omega(Y) := \det\left[i\omega-(\alpha+i\beta)-Y\right]=0\label{eq:Y_S}
\end{align}
has a single root 
\begin{align*}
    Y=i\omega-\alpha-i\beta.
\end{align*}
Hence the ACS reads 
\begin{equation}
\gamma(\omega)= -\ln|Y(\omega)|=-\frac{1}{2}\ln\left[ \alpha^2+(\omega-\beta)^2\right].\label{eq:ACS_S}
\end{equation}
We can see that the asymptotic continuous spectrum implies instability for $|\alpha| < 1$ \cite{Wolfrum2006}, see Fig. \ref{Fig:NonDDE}.
\begin{figure}[H]
		\centering		\subfigure{\includegraphics[width=7cm]{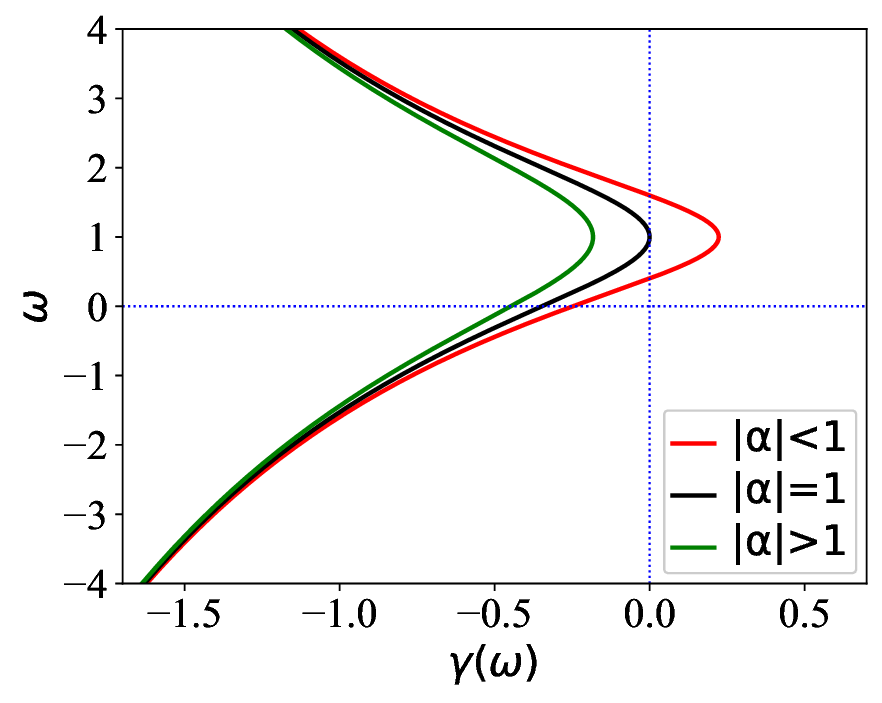}}
		\caption{ACS for Eq.~\eqref{eq:linearSL} with $\beta=1$  and different $\alpha$: unstable ($\alpha=0.8$), critical ($\alpha=1$), and stable  ($\alpha=1.2$). }
		\label{Fig:NonDDE}
	\end{figure}
 
 In this case, system \eqref{eq:linearSL} is an example of the  \textit{universality class II} DDE under the conditions of $|\alpha|<1$ and $\alpha^2+\beta^2>1$. 
 Since we consider the system in a complex form, only one curve of ACS is present.
If we had used the real form with two variables $\Re(z)$ and $\Im(z)$, we would obtain two complex conjugated curves of the ACS as in Fig.~\ref{fig:classII-1}(b).

The critical frequency $\omega_{H}$ is given by the root of
$
\gamma\left(\omega_{H}\right)=-\frac{1}{2}\ln\left[\alpha^{2}+(\omega_{H}-\beta)^{2}\right]=0,
$
leading to 
\begin{align*}
    \omega_{H_m}=\beta\pm\sqrt{1-\alpha^{2}}, \quad m=1,2.
\end{align*}
The phase $\phi_{H}$ is given as 
\[
\phi_{H_m}=-\arg\left[\pm i\sqrt{1-\alpha^{2}}-\alpha\right], \quad m=1,2.
\]

Corollary~\ref{thm:4.1} for the class II DDEs implies that there are transverse crossings of the characteristic roots of  \eqref{eq:CharaE} with $\pm i\omega_{H_m}$  occur for $|\alpha|<1$ and the following time delays: 
\begin{align}
  \tau^{(m)}_{k}=\frac{1}{\omega_{H_m}}(\phi_{H_m}+2\pi k) =-\frac{1}{\beta\pm\sqrt{1-\alpha^2}}\left(\arg\left[\pm i\sqrt{1-\alpha^{2}}-\alpha\right]+2\pi k\right),\quad k=0,1,2,\cdots.\label{eq:tauk_scalar_2}
\end{align}
Moreover, $\tau^{(2)}_{k}$ are stabilizing and  $\tau^{(1)}_{k}$ destabilizing.
The time-delay values $\tau^{(m)}_{k}$ correspond to the stabilizing and destabilizing Hopf bifurcations of the Stuart-Landau system, respectively.  $i\omega_{H_{m}}$ are the frequencies of the Hopf bifurcations. 
Figure~\ref{Fig:Spec} illustrates the spectrum for these bifurcation moments. 



	\begin{figure}[H]
		\centering		\subfigure{\includegraphics[width=7.5cm]{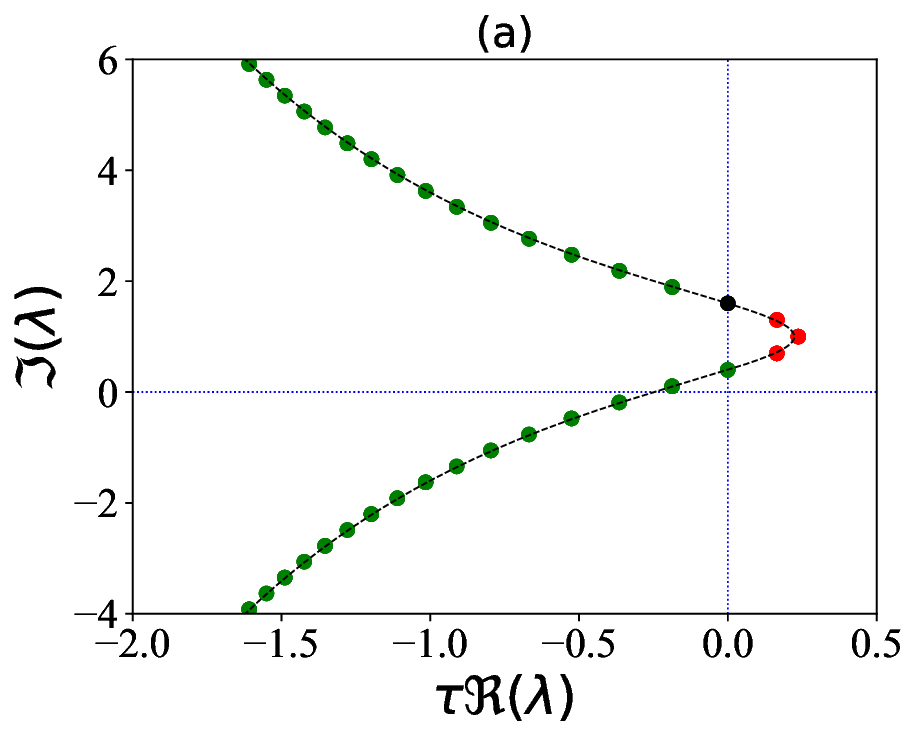}}\subfigure{\includegraphics[width=7.5cm]{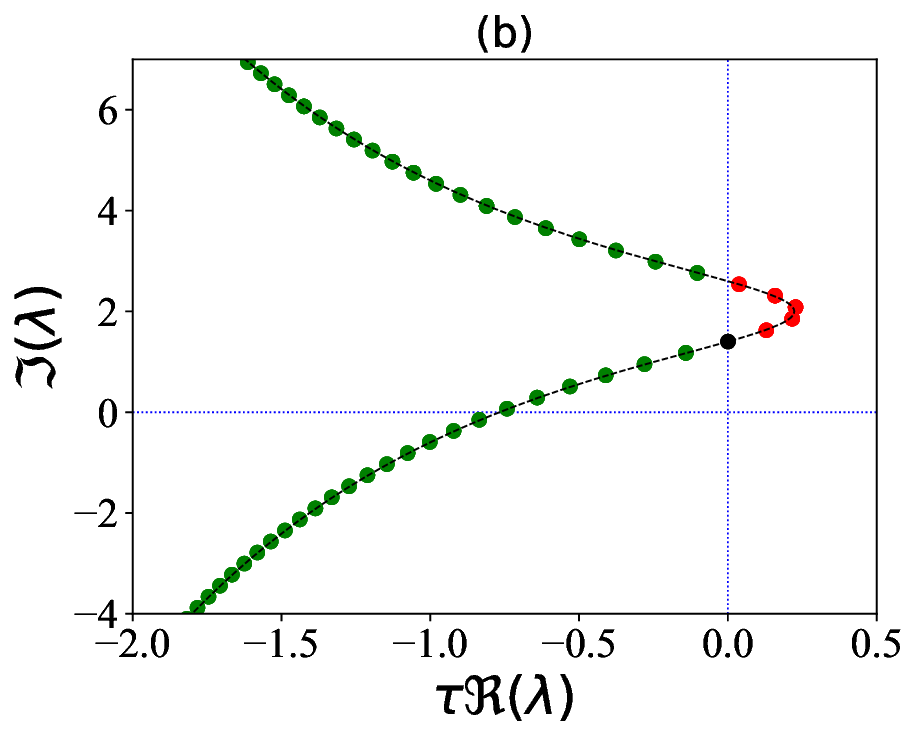}}
		\caption{Spectrum and the ACS for time-delayed system \eqref{eq:SL} for the following two cases: (a) $\alpha=0.8$, $\beta=1$, and $\tau=\tau_{6}^{(1)}= 21.989$; (b) $\alpha=0.8$, $\beta=2$, and  $\tau=\tau_{6}^{(2)}= 28.7$.}
		\label{Fig:Spec}
	\end{figure}

\subsection{Bifurcating periodic solutions and their stability}\label{subsec6-2} 
We now describe the bifurcation scenario with increasing $\tau$. If the ACS crosses the imaginary axis (the case of universality class II), Hopf bifurcations occur and lead to the appearance of periodic solutions. 
Due to the equivariance of \eqref{eq:SL} with respect to the transformation $z\to z e^{i\phi}$ for any $\phi\in S^1$, typical periodic solutions in \eqref{eq:SL} have the form $z(t)=ae^{i\omega t}$ with the amplitude $a$ and frequency $\omega$. 
Substituting this Ansatz into equation \eqref{eq:SL}, we obtain
\begin{align*}
    i\omega_{H}=\alpha+i\beta-a^2+e^{-i\omega\tau},
\end{align*}
or, equivalently,
\begin{align}
 \label{eq:a}
    a&=\sqrt{\alpha+\cos\left(\omega\tau\right)},\\
    \omega&=\beta-\sin\left(\omega\tau\right).\label{eq:tau}
\end{align}
Using \eqref{eq:a}--\eqref{eq:tau}, the amplitudes $a$ versus time-delay $\tau$ of the periodic solutions can be parametrically represented as $a(\phi)$ and $\tau(\phi)$, where
\begin{align}
\label{eq:a_phi}
    a(\phi)&=\sqrt{\alpha+\cos\left(\phi\right)},\\
    \tau(\phi)&=\frac{\phi+2\pi k}{\beta-\sin\left(\phi\right)}, \quad k\in 0,1,2,\dots\label{eq:tau_phi}
\end{align}

\paragraph{Case $|\alpha|<1$} In this case  DDE \eqref{eq:SL} is of universality class II.
 Figure \ref{Fig:Bif} shows the resulting diagrams of the amplitudes $a$ with respect to time delay $\tau$. 
 We observe the Hopf bifurcations as predicted by Eq.~\eqref{eq:tauk_scalar_2}.
The Hopf bifurcations also correspond to the condition $a=0$, i.e., $\alpha+\cos(\phi_m)=0$, $m=1,2$, and 
\begin{equation}
\tau_k^{(m)}=\frac{\phi_m+2\pi k}{\beta-\sin(\phi_m)}, k=0,1,2,\dots,
\end{equation}
where $\omega_m = \beta-\sin(\phi_m)$ is the Hopf frequency, see Eq.~\eqref{eq:tauk_scalar_2}.
The periodic solutions exist for $\alpha+\cos\left(\phi\right)>0$.
Moreover, as the Hopf bifurcations at $\tau_k^{(1)}$ are destabilizing,  the branch of periodic solutions is emerging from these points $(a=0,\tau=\tau_k^{(1)})$ in the bifurcation diagram.  Accordingly, the Hopf bifurcations at $\tau_k^{(2)}$ are stabilizing, and the branches are terminating there.

We note additionally that the branches of periodic solutions are bounded in $\tau$ for $\beta>1$. Hence, they form so-called bridges of periodic solutions connecting the corresponding Hopf bifurcation points as in Fig.~\ref{Fig:Bif}(a) and (b), see also \cite{Pieroux2001}. 
However, for $\beta= 1$, all branches are disconnected as in Fig.~\ref{Fig:Bif}(c). This can be understood from Eq.~\eqref{eq:tau_phi}, since the denominator approaches zero for some values of $\phi$. 
\begin{figure}[H]
\centering		
\subfigure{\includegraphics[width=7.5cm]{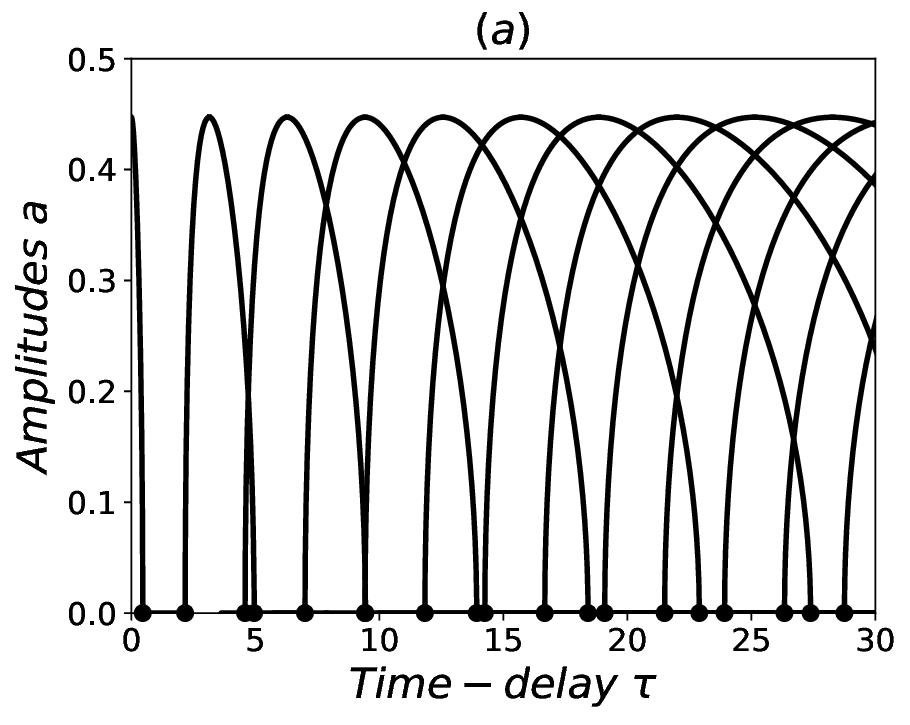}} 
\subfigure{\includegraphics[width=7.5cm]{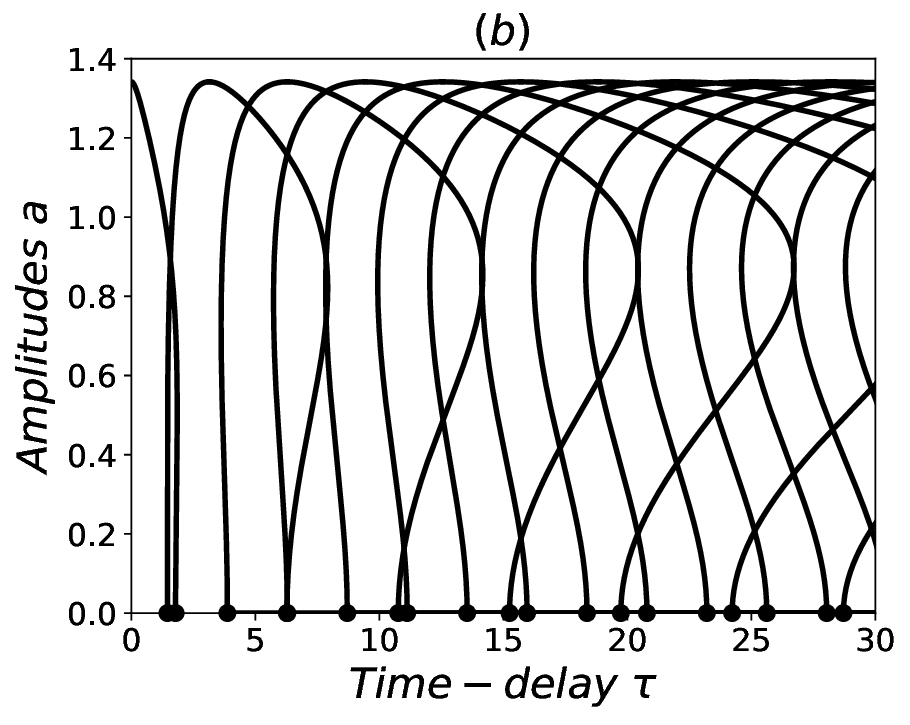}}
\subfigure{\includegraphics[width=7.5cm]{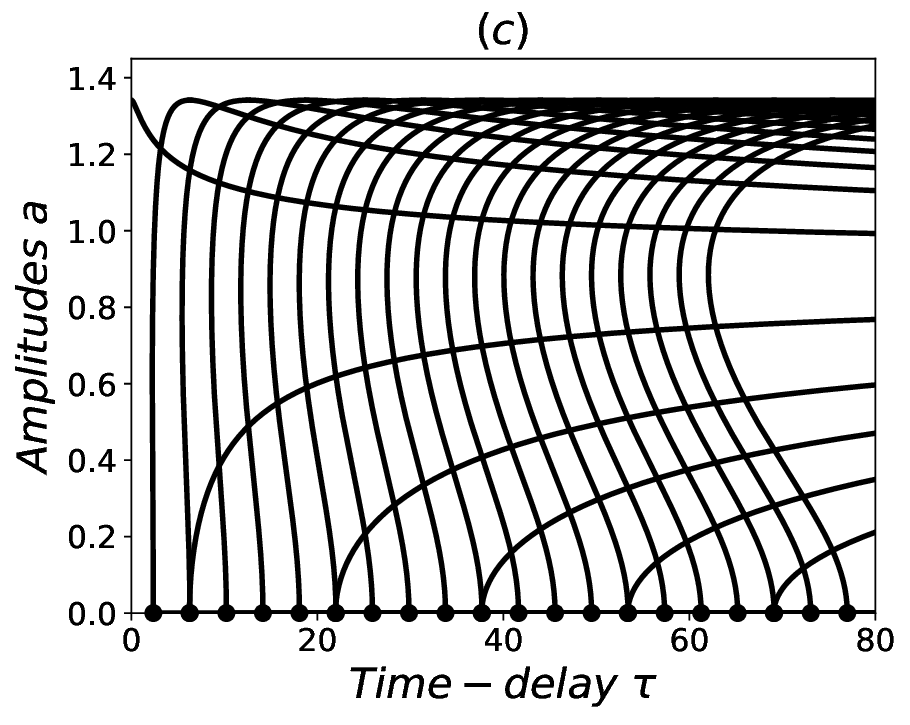}}
\caption{Dependence of the amplitude $a$  on the bifurcation parameter $\tau$ with $|\alpha|<1$. The parameters of the equations \eqref{eq:a_phi} and \eqref{eq:tau_phi} are respectively: (a) $\alpha=-0.8$, $\beta=2$; (b) $\alpha=0.8$, $\beta=2$; (c) $\alpha=0.8$, $\beta=1$. The branches of periodic solutions are bounded in $\tau$ for $\beta>1$, but all branches are disconnected for $\beta = 1$.}\label{Fig:Bif}
\end{figure}

\paragraph{Case $-\infty < \alpha < -1$} This case corresponds to the absolute (delay-independent) stability \cite{Yanchuk2022b} of the equilibrium $z=0$ and, hence, there are no Hopf bifurcations for positive delays $\tau$. 
If a periodic solution exists in this case, it can nether be continued to $\tau=0$ nor to any Hopf bifurcation with the fixed point. 

\paragraph{Case $\alpha\ge 1$}
For this case, there is one periodic solution for $\tau=0$, but no Hopf bifurcations for positive delays, see Fig.~\ref{Fig:Bif2}(a) and (b).
\begin{figure}[H]
\centering		
\subfigure{\includegraphics[width=7.5cm]{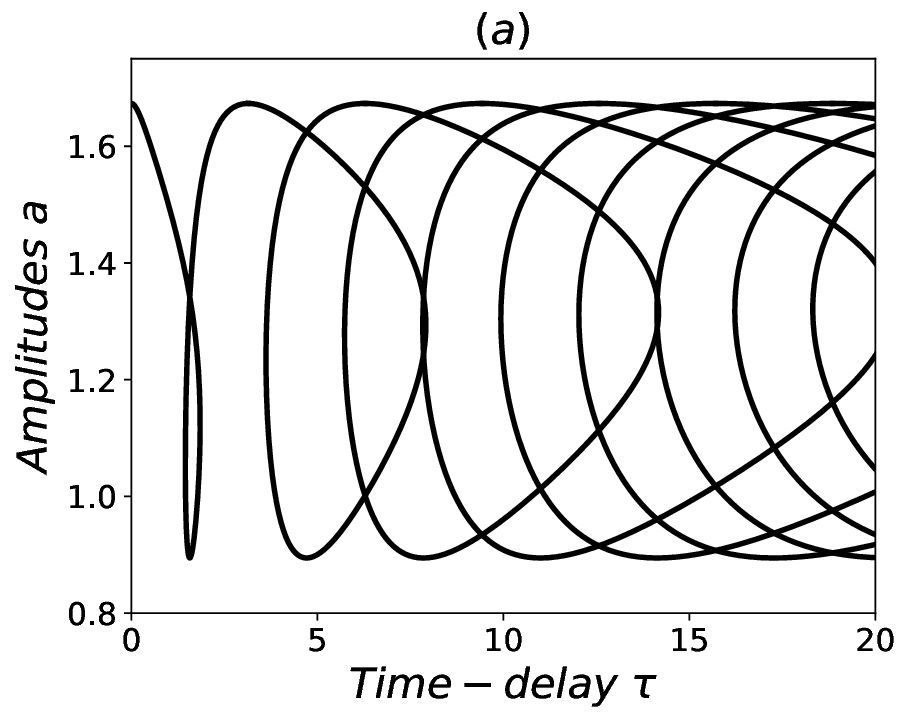}} 	
\subfigure{\includegraphics[width=7.5cm]{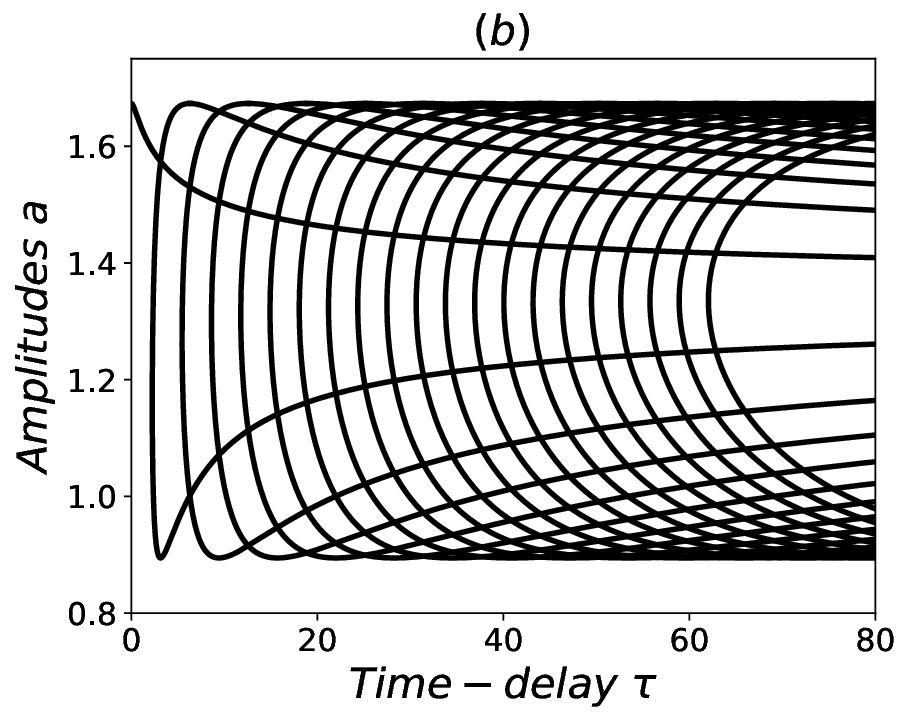}} 
\caption{Dependence of the amplitude $a$  on the bifurcation parameter $\tau$ with $\alpha\ge 1$. The parameters of the equations \eqref{eq:a_phi} and \eqref{eq:tau_phi} are respectively:(a) $\alpha=1.8$, $\beta=2$; (b) $\alpha=1.8$, $\beta=1$.  There are periodic solutions but no Hopf bifurcations for positive delays. The periodic solutions are connected for $\beta>1$ but disconnected for $\beta = 1$.}\label{Fig:Bif2}
	\end{figure}

\section{Conclusions}\label{sec8} 

We show that linear DDEs with one delay possess universal destabilization / stabilization scenarios as time-delay is varied. These scenarios lead to universal cascades of Hopf bifurcations in the corresponding nonlinear DDEs. 
The universality classes, corresponding to the same scenario, are determined with the help of the asymptotic continuous spectrum, which can be rather calculated.  

Although the number of potential universality classes is unlimited, we 
determine three most common universality classes (I to III) and describe their bifurcation scenarios. The universality classes are naturally extended by the class ``0" of absolutely hyperbolic DDEs introduced in \cite{Yanchuk2022b}, which do not exhibit any bifurcations for any value of delays. 

All our results are illustrated by numerous simulation examples. In addition, we present an example of the nonlinear Stuart-Landau model with time-delayed feedback and show how the ``bridges" or periodic orbits connect the stabilizing and destabilizing Hopf bifurcations. 

We believe that our results can be useful for studying many applied problems where the time-delay appears as a parameter. A possible interesting extension could be the study of time-varying or state-dependent delays, where the delay would traverse through a cascade of Hopf bifurcation, which we have described here.

\bibliographystyle{elsarticle-num}  

\end{document}